\documentclass[11pt,english,british,refpage,intoc,bibliography=totoc,index=totoc,BCOR=7.5mm,captions=tableheading]{extarticle}
\usepackage{mathptmx}
\usepackage[T1]{fontenc}
\usepackage[latin9]{inputenc}
\usepackage[a4paper]{geometry}
\usepackage{color}
\usepackage{babel}
\usepackage{mathrsfs}
\usepackage{amsmath}
\usepackage{amsthm}
\usepackage{amssymb}
\usepackage[numbers]{natbib}
\usepackage[unicode=true,pdfusetitle,
 bookmarks=true,bookmarksnumbered=true,bookmarksopen=false,
 breaklinks=false,pdfborder={0 0 0},pdfborderstyle={},backref=false,colorlinks=true]
 {hyperref}
\hypersetup{
 linkcolor=black, citecolor=blue, urlcolor=black, filecolor=blue, pdfpagelayout=OneColumn, pdfnewwindow=true, pdfstartview=XYZ, plainpages=false}

\makeatletter
\numberwithin{equation}{section}
\numberwithin{figure}{section}
\theoremstyle{plain}
\newtheorem{thm}{\protect\theoremname}[section]
  \theoremstyle{plain}
  \newtheorem{lem}[thm]{\protect\lemmaname}
  \theoremstyle{remark}
  \newtheorem{rem}[thm]{\protect\remarkname}
  \theoremstyle{plain}
  \newtheorem{cor}[thm]{\protect\corollaryname}

\@ifundefined{date}{}{\date{}}
\usepackage{caption}
\usepackage[nottoc]{tocbibind}
\usepackage{relsize}
\usepackage{dsfont}
\DeclareMathAlphabet{\mathcal}{OMS}{cmsy}{m}{n}
\allowdisplaybreaks[4]

\makeatother

  \addto\captionsbritish{\renewcommand{\corollaryname}{Corollary}}
  \addto\captionsbritish{\renewcommand{\lemmaname}{Lemma}}
  \addto\captionsbritish{\renewcommand{\remarkname}{Remark}}
  \addto\captionsbritish{\renewcommand{\theoremname}{Theorem}}
  \addto\captionsenglish{\renewcommand{\corollaryname}{Corollary}}
  \addto\captionsenglish{\renewcommand{\lemmaname}{Lemma}}
  \addto\captionsenglish{\renewcommand{\remarkname}{Remark}}
  \addto\captionsenglish{\renewcommand{\theoremname}{Theorem}}
  \providecommand{\corollaryname}{Corollary}
  \providecommand{\lemmaname}{Lemma}
  \providecommand{\remarkname}{Remark}
\providecommand{\theoremname}{Theorem}

\begin{document}
\global\long\def\fracap{\genfrac(){0.5pt}{1}{\alpha}{p}}

\global\long\def\lchoosej{\genfrac(){0pt}{1}{l}{j}}

\title{Fine properties of fractional Brownian motions \\ on Wiener space}

\author{Jiawei Li\thanks{Mathematical Institute, University of Oxford, Oxford OX2 6GG. Email:
\protect\href{mailto:jiawei.li@maths.ox.ac.uk}{jiawei.li@maths.ox.ac.uk}} \ and Zhongmin Qian\thanks{Mathematical Institute, University of Oxford, Oxford OX2 6GG. Email:
\protect\href{mailto:zhongmin.qian@maths.ox.ac.uk}{zhongmin.qian@maths.ox.ac.uk}}}
\maketitle
\begin{abstract}
We study several important fine properties for the family of fractional
Brownian motions with Hurst parameter $H$ under the $(p,r)$-capacity
on classical Wiener space introduced by Malliavin. We regard fractional
Brownian motions as Wiener functionals via the integral representation
discovered by Decreusefond and \"{U}st\"{u}nel, and show non differentiability,
modulus of continuity, law of iterated Logarithm(LIL) and self-avoiding
properties of fractional Brownian motion sample paths using Malliavin
calculus as well as the tools developed in the previous work by Fukushima,
Takeda and etc. for Brownian motion case.
\end{abstract}
\selectlanguage{english}%

\medskip

\selectlanguage{british}%

\section{Introduction}

Fractional Brownian motions (fBMs for simplicity), as archetypical
examples of Gaussian processes have attracted researchers in recent
years. The stochastic calculus and sample path properties for them
are mainly studied in the setting of Gaussian measures (the Malliavin
calculus for example) and Gaussian processes. In this article, we
explore the fine properties of fBMs as measurable functions on the
Wiener space. By fine properties here we mean those sample properties
which are measured uniformly by the capacities associated with the
classical Wiener space.

Recall that an fBM, $(B_{t})_{t\geq0}$ with Hurst parameter $H\in(0,1)$
is, by definition, a centred Gaussian process with its co-variance
function given by
\[
R(t,s)=\mathbb{E}\left[B_{t}B_{s}\right]=\frac{1}{2}\left(t^{2H}+s^{2H}-|t-s|^{2H}\right)
\]
for $s,t\geq0$. FBMs were firstly introduced by Kolmogorov \citep{Kolmogorov1940b}
in early 1940s, which were named as fractional Brownian motion by
Mandelbrot and Van Ness \citep{Mandelbrot1968} in 1968. An integral
representation for fBM with Hurst parameter $H$ was discovered in
\citep{Mandelbrot1968}, which is given by 
\[
B_{t}=\frac{1}{\sqrt{C(H)}}\left\{ \int_{-\infty}^{0}\left[(t-s)^{H-\frac{1}{2}}-(-s)^{H-\frac{1}{2}}\right]dW_{s}+\int_{0}^{t}(t-s)^{H-\frac{1}{2}}dW_{s}\right\} ,
\]
where $(W_{t})$ is a standard two-sided Brownian motion, and 
\[
C(H)=\int_{-\infty}^{0}\left[(1-s)^{H-\frac{1}{2}}-(-s)^{H-\frac{1}{2}}\right]^{2}ds+\frac{1}{2H}.
\]

The sample paths properties of fBMs, like all other aspects of their
laws, depend crucially on the Hurst parameter $H$. FBM with Hurst
parameter $H=\frac{1}{2}$ is just a standard Brownian motion. The
study of sample paths of Brownian motion has been one of the primary
components in stochastic analysis, see e.g. It\^{o}-McKean \citep{Ito1974},
Karatzas-Shreve \citep{Karatzas2012}, Revuz-Yor \citep{Revuz2013}
and other excellent references there-in. FBMs have stationary increments,
unlike Brownian motion however, the increments of fBMs are no longer
independent in the case where $H\neq\frac{1}{2}$. If $H>\frac{1}{2}$,
the increments over different time intervals are positively correlated,
while for $H<\frac{1}{2}$, the increments are negatively correlated.
fBMs are self-similar Gaussian processes with long time memory if
$H\neq\frac{1}{2}$, which are neither Markov processes, nor semi-martingales.
Decreusefond and \"{U}st\"{u}nel \citep{Decreusefond1999} identified
the Cameron-Martin spaces of fBMs, and deduced another form of representations
for fBMs, in terms of Wiener integrals with respect to Brownian motion,
and thus realised fBMs as measurable functionals of Brownian motion.
FBMs are examples of Wiener functionals which are not solutions to
It\^{o}'s stochastic differential equations.  The advantage of considering
fBMs as Wiener functionals lies in the fact that one may derive results
for fBMs with different Hurst parameters in terms of concepts defined
by Brownian motion, such as capacities. In this paper we derive several
sample properties of fBMs with respect to the capacities defined on
the classical Wiener space by the standard Brownian motion, rather
than on different Gaussian spaces induced by fBMs with different Hurst
parameters. We prove a few interesting fine properties for the family
of fBMs with respect to the $(p,r)$-capacity defined in the sense
of Malliavin \citep{Malliavin1984} on the classical Wiener space.
To be more specific, we will study non-differentiability, modulus
of continuity, law of iterated logarithm and self-intersection of
fBMs measured by capacities on the classical Wiener space. These sample
path properties have been investigated over past few decades, for
both Brownian motion and fBMs, even for general Gaussian processes,
under both probability and $(p,r)$-capacity, see for example \citep{Cohen2013,Karatzas2012,Revuz2013}.
There is a huge amount of literature on this aspect. Paley, Wiener
and Zygmund \citep{Paley1933} showed the almost everywhere non-differentiability
of Brownian motion sample paths (see also the argument by Dvoretzky,
Erd\H{o}s and Kakutani in \citep{Dvoretzky1961}), and Mandelbrot
and Van Ness \citep{Mandelbrot1968} proved that fBM sample paths
are also non-differentiable almost surely. For the modulus of continuity,
L\'{e}vy \citep{Levy1937} established the result on H\"{o}lder
continuity for Brownian motion. In \citep{Decreusefond1999}, it was
shown by Decreusefond and \"{U}st\"{u}nel that sample paths of fBM
with Hurst parameter $H$ are almost surely H\"{o}lder continuous
only of order less than $H$. Khintchine \citep{Khintchine1933a}
extended the law of iterated logarithm from the case of random walk
to Brownian motion. In \citep{Coutin2007}, Coutin \citep{Coutin2007}
mentioned the following result on the law of iterated logarithm for
fBM
\begin{equation}
\limsup_{\varepsilon\to0^{+}}\frac{B_{t+\varepsilon}-B_{t}}{\sqrt{2\varepsilon^{2H}\log\log(1/\varepsilon)}}=1,\quad\text{a.s.}\label{eq:fBM_LIL_prob}
\end{equation}
while, to the best knowledge of the present authors, a written proof
doesn't exist for the case that $H<\frac{1}{2}$, but see e.g. \citep{Arcones1995}
for the functional version of the law of iterated logarithm for Gaussian
processes. For the case that $H\in(0,\frac{1}{2}]$, this was established
in Cohen and Istas \citep{Cohen2013}. Whether a sample path of one
stochastic process intersects itself has been an appealing problem
due to its connection with statistical field theory (see e.g. Itzykson-Drouffe
\citep{Itzykson1989}). It dates back to 1944 when Kakutani \citep{Kakutani1944}
answered this question for Brownian motion. He demonstrated that $d$-dimensional
Brownian motion is self-avoiding when $d\geq5$, and his solution
was accomplished in his joint work with Dvoretzky and Erd\H{o}s \citep{Dvoretzky1950}
showing that $d=4$ is the optimal dimensional for this property.
One can show that, when $d>\frac{2}{H}$, with probability one $(B_{t})_{t\geq0}$
has no double point almost surely by using the classical argument
see e.g. Kakutani \citep{Kakutani1944}. There is little information
on the optimal dimension for self-avoiding property for fractional
Brownian motion case due to the loss of potential theory. In early
1980s, Fukushima \citep{Fukushima1980} introduced the capacity defined
via Dirichlet forms, which is equivalent to $(p,r)$-capacity given
by Malliavin \citep{Malliavin1984} with $r=1$ and $p=2$, and proved
all above sample path properties for Brownian motion with respect
to this capacity. Malliavin \citep{Malliavin1984} introduced the
$(p,r)$-capacity defined via Malliavin derivatives for subsets of
the Wiener space, and Takeda \citep{Takeda1984} extended Fukushima's
result for Brownian motion to the case of $(p,r)$-capacity. Fukushima
\citep{Fukushima1984} also showed the absence of double points under
$(2,1)$-capacity for $d$-dimensional Brownian motion when $d\geq7$,
and later Lyons \citep{Lyons1986} determined the critical dimension
$d=6$ for the absence of double points, by using potential theory
of Brownian motion. Inspired by the argument in \citep{Fukushima1984}
and \citep{Takeda1984}, we will derive similar results for the family
of fBMs with different Hurst parameters $H$ with respect to one uniform
capacity on classical Wiener space. These results describe better
the behaviour of sample paths for fBM as they remain true when $H$
varies. 

The quasi-sure analysis, initiated and created mainly by Malliavin
(see e.g. \citep{Malliavin1978,Malliavin1978a,Malliavin1984,Malliavin2015}),
Fukushima, Watanabe and etc. \citep{Fukushima1980,Fukushima1984,Fukushima1993,Watanabe1984},
is the research area whose main feature is to study various Wiener
functionals (whose laws are typically mutually singular such as Brownian
motion and Brownian bridge). In the past, the majority of Wiener functionals
considered in literature are the solutions of It\^{o}'s stochastic
differential equations, for which It\^{o}'s stochastic calculus and
the potential theory for diffusion processes may be utilised to study
their fine properties. In this article, we take the point-view that
fBMs are typical Wiener functionals, i.e. measurable functionals of
Brownian motion, for which traditional tools such as Markovian or
It\^{o}'s calculus are no longer applicable. In order to derive sample
properties of fBMs in terms of capacities of Brownian motion, we employ
the basic techniques developed by Malliavin, Fukushima, Takeda and
etc. during last decades and adopted their fundamental ideas to our
study. While we have to overcome several difficulties, which were
mainly achieved by carefully controlling the Malliavin derivatives
of fBMs. 

The paper is organised as the following. In Section 2, we introduce
definitions and notations related to classical Wiener capacities and
fractional Brownian motion. In section 3, we establish the modulus
of continuity result following the argument by Fukushima \citep{Fukushima1984},
and hence deduce the quasi-surely H\"{o}lder continuity of fBMs regarded
as Wiener functionals. This allows us to take a continuous modifications
of fBMs and prove non-differentiability in section 4 based on the
argument by Dvoretzky, Erd\H{o}s and Kakutani in \citep{Dvoretzky1961},
as well as the law of iterated logarithm (LIL) when $p=2$ and $r=1$
with restriction $H\leq\frac{1}{2}$ in section 5. Finally, in section
6, we prove the self-avoiding property of $d$-dimensional fBMs under
$c_{2,1}$ when $d>\frac{2}{H}+2$ and $H\leq\frac{1}{2}$. 

\section{Wiener functionals}

The Wiener measure is by definition the distribution of Brownian motion,
which defines in turn the Wiener space, a convenient framework for
the study of Wiener functionals (see e.g. Chapter V Section 8, Ikeda
and Watanabe\citep{Ikeda2014}). Let $\boldsymbol{W}_{0}^{d}$ denote
the space of all continuous paths in the Euclidean space $\mathbb{R}^{d}$,
started at the origin. $\boldsymbol{W}_{0}^{d}$ is a complete separable
Banach space under the norm
\[
\lVert\omega\rVert=\sum_{n=1}^{\infty}2^{-n}\max_{0\leq t\leq n}\left|\omega(t)\right|,
\]
which induces the topology of uniform convergence over every compact
subset of $[0,\infty)$. The Borel $\sigma$-algebra on $\boldsymbol{W}_{0}^{d}$
is denoted by $\mathscr{B}(\boldsymbol{W}_{0}^{d})$ or by $\mathscr{B}$
if no confusion may arise. Following It\^{o} and McKean \citep{Ito1974},
we will use $\omega$ to denote a general element, so that $\omega(t)$
is the value of a sample path $\omega$ at $t\geq0$, the $t$-th
coordinate of a sample point $\omega\in\boldsymbol{W}_{0}^{d}$. The
same notation $\omega(t)$ denotes also the coordinate mapping $\omega\rightarrow\omega(t)$,
and the parametrised family $\left\{ \omega(t):t\geq0\right\} $ is
the coordinate process on $\boldsymbol{W}_{0}^{d}$. The coordinate
mapping $\omega(t)$ may be denoted by $\omega_{t}$ (for $t\geq0$)
too. Then the Borel $\sigma$-algebra $\mathscr{B}\left(\boldsymbol{W}_{0}^{d}\right)$
is the smallest $\sigma$-algebra on $\boldsymbol{W}_{0}^{d}$ with
which all coordinate functions $\omega(t)$ (for $t\geq0$) are measurable
(for a proof, see e.g. Stroock and Varadhan \citep{Stroock2007}).
The Wiener measure $P^{W}$ is the unique probability on $\left(\boldsymbol{W}_{0}^{d},\mathscr{B}\right)$
such that the coordinate process $(\omega(t))_{t\geq0}$ of $\boldsymbol{W}_{0}^{d}$
is a standard Brownian motion in $\mathbb{R}^{d}$. To complete the
definition of the classical Wiener space, one should identify the
Cameron-Martin space of the Wiener measure $P^{W}$. To this end,
it is better to identify the Wiener measure $P^{W}$ as a Gaussian
measure on $\boldsymbol{W}_{0}^{d}$. For simplicity, $\boldsymbol{W}_{0}^{d}$
and $P^{W}$ will be denoted by $\boldsymbol{W}$ and $P$ respectively,
if no confusion is possible. 

Let $\mathcal{H}$ be the space of all $h\in\boldsymbol{W}$ such
that $t\rightarrow h(t)$ is absolutely continuous and its generalized
derivative $\dot{h}$ is square-integrable on $[0,\infty)$. $\mathcal{H}$
is a Hilbert space under the norm $\Vert h\Vert_{\mathcal{H}}=\sqrt{\int_{0}^{\infty}\vert\dot{h}(t)\vert^{2}dt}$,
and the dual space $\boldsymbol{W}^{\star}$ of all continuous linear
functionals on $\boldsymbol{W}$ can be identified as a subset of
$\mathcal{H}$, so that we have the continuous densely embedding $\boldsymbol{W}^{\star}\hookrightarrow\mathcal{H}\hookrightarrow\boldsymbol{W}$
with respect to their corresponding norms. 

$P$ is the unique measure on $\left(\boldsymbol{W},\mathscr{B}\right)$
such that every continuous linear functional $\gamma\in\boldsymbol{W}^{\star}$
has a normal distribution with mean zero and variance $\Vert\gamma\Vert_{\mathcal{H}}^{2}$.
In other words, $P$ is the unique probability measure on $\boldsymbol{W}$
such that
\[
\int_{\boldsymbol{W}}e^{i\gamma(\omega)}P(d\omega)=\exp\left[-\frac{1}{2}\Vert\gamma\Vert_{\mathcal{H}}^{2}\right]
\]
for every $\gamma\in\boldsymbol{W}^{\star}$. Therefore, every $h\in\mathcal{H}$
corresponds (unique up to almost surely) to a random variable on $\boldsymbol{W}$,
still denoted by $h$, which has a normal distribution $N(0,\Vert h\Vert_{\mathcal{H}}^{2})$.
In fact, for every $h\in\mathcal{H}$, the corresponding Gaussian
variable $h$ can be identified with the It$\hat{o}$ integral, denoted
by $\left[h\right]$, $\int_{0}^{\infty}\dot{h}d\omega$ of $\dot{h}$
against the Brownian motion $(\omega(t))_{t\geq0}$, which is defined
in probability sense. Under this sense, the triple $(\boldsymbol{W},\mathcal{H},P)$
is an example of abstract Wiener spaces, a concept introduced by L.
Gross \citep{Gross1967}, called the classical Wiener space. The completion
of the Borel $\sigma$-algebra $\mathscr{B}$ is denoted by $\mathscr{F}$. 

An $\mathscr{F}$-measurable (valued in a separable Hilbert space)
function on $\boldsymbol{W}$ is called, according to the convention
in literature, a Wiener functional. 

\subsection{Malliavin derivative and capacity}

A differential structure on the Wiener space $(\boldsymbol{W},\mathcal{H},P)$
compatible to the Wiener measure was introduced by Malliavin\citep{Malliavin1978},
\citep{Malliavin1978a}. The Malliavin derivative for smooth random
variables of form
\[
F=f([h_{1}],,\cdots,[h_{n}]),\quad h_{i}\in\mathcal{H},
\]
can be defined formally by differentiating $F$, as long as $f\in C_{p}^{\infty}(\mathbb{R}^{n})$,
a function whose partial derivatives have polynomial growth. The Malliavin
derivative of $F$ is an $\mathcal{H}$-valued random variable defined
by 
\[
DF=\sum_{i=1}^{n}\partial_{i}f([h_{1}],,\cdots,[h_{n}])h_{i},
\]
where $\partial_{i}f(x_{1},\cdots,x_{n})$ is the partial derivative
of $f$ in $i$-th component. The high order Malliavin derivatives
$D^{k}F$ for all $k\geq1$ may be defined inductively. The collection
of all such smooth random variables $F$ is denoted by $\mathcal{S}$.
For $r\in\mathbb{N}$ and $1<p<\infty$, let $\mathbb{D}_{r}^{p}$
be the completion of $\mathcal{S}$ with respect to the Sobolev norm
\[
\lVert F\rVert_{\mathbb{D}_{r}^{p}}=\left(\mathbb{E}\left[|F|^{p}\right]+\sum_{k=1}^{r}\mathbb{E}\left[\left|\lVert D^{k}F\rVert_{\mathcal{H}^{\otimes k}}\right|^{p}\right]\right)^{1/p}.
\]
The $(p,r)$-capacity of an open subset $O$ of $\boldsymbol{W}$
is defined by (see e.g. \citep{Malliavin2015}):

\[
c_{p,r}\left(O\right)=\inf\left\{ \lVert\varphi\rVert_{\mathbb{D}_{r}^{p}}:\varphi\in\mathbb{D}_{r}^{p},\ \varphi\geq1\ \text{a.e. on }O,\ \varphi\geq0\ \text{a.e. on }\boldsymbol{W}\right\} ,
\]
and for an arbitrary subset $A$ of $\boldsymbol{W}$, its $(p,r)$-capacity
is
\[
c_{p,r}\left(A\right)=\inf\left\{ c_{p,r}\left(O\right):A\subset O,\ O\ \text{ is open}\right\} .
\]
$A\subset\boldsymbol{W}$ is said to be slim if $c_{p,r}(A)=0$ for
all $r\in\mathbb{N}$ and $1<p<\infty$. A property $\pi$ defined
over $\boldsymbol{W}$ is said to hold quasi-surely (q.s.) if the
set on which this property is not satisfied is slim. 

The notion of slim sets on the classical Wiener space $(\boldsymbol{W},\mathcal{H},P)$
can be studied via the Orenstein-Uhlenbeck operator, which gives rise
to a different but equivalent approach to $(p,r)$-capacity. For a
given $p\in\left[1,\infty\right]$, let $(T_{t})_{t\geq0}$ denote
the Ornstein-Uhlenbeck semi-group on $L^{p}(\boldsymbol{W},P)$, which
is the one-parameter semi-group of contractions on $L^{p}(\boldsymbol{W},P)$
given by
\[
T_{t}u(x)=\int_{\boldsymbol{W}}u\left(e^{-t}x+\sqrt{1-e^{-2t}}\omega\right)P(d\omega).
\]
Let $L$ be the generator of the semi-group $(T_{t})$, that is,
\begin{align*}
\mathcal{D}(L) & =\left\{ u\in L^{p}(\boldsymbol{W},P):\lim_{t\downarrow0}\frac{T_{t}u-u}{t}\textrm{ exists in }L^{p}\textrm{-space}\right\} 
\end{align*}
and
\[
Lu=\lim_{t\downarrow0}\frac{T_{t}u-u}{t}\textrm{ for }u\in\mathcal{D}(L).
\]
For each $r>0$, $(I-L)^{-\frac{r}{2}}$ is again a contraction on
$L^{p}(\boldsymbol{W},P)$, and is given by the following integral
\[
(I-L)^{-\frac{r}{2}}=\frac{1}{\Gamma(r/2)}\int_{0}^{\infty}t^{\frac{r}{2}-1}e^{-t}T_{t}dt
\]
(defined in the sense of Bochner's integrals). The corresponding Sobolev
norm $\lVert\cdot\rVert_{r,p}$ (where $1<p<\infty$) is then defined
by
\[
\lVert u\rVert_{r,p}=\lVert(I-L)^{-\frac{r}{2}}u\rVert_{p}.
\]

The corresponding $(p,r)$-capacity $C_{r,p}$, following Fukushima's
convention in \citep{Fukushima1993}, can be defined in a similar
manner as before, namely, for an open subset $O$ of $\boldsymbol{W}$,
\[
C_{r,p}\left(O\right)=\inf\left\{ \lVert\phi\rVert_{p}^{p}:\ (I-L)^{-\frac{r}{2}}\phi\geq1\ \text{a.e. on }O,\ (I-L)^{-\frac{r}{2}}\phi\geq0\ \text{a.e. on }\boldsymbol{W}\right\} ,
\]
(with convention that $\inf\textrm{�}=\infty$) and 
\[
C_{r,p}\left(A\right)=\inf\left\{ \mathrm{C}_{r,p}\left(O\right):A\subset O,\ O\ \text{ is open}\right\} 
\]
for an arbitrary subset $A$ of $\boldsymbol{W}$. It was Meyer \citep{Meyer1983}
who proved that norms $\lVert\cdot\rVert_{\mathbb{D}_{r}^{p}}$ and
$\lVert\cdot\rVert_{r,p}$ are equivalent, and it follows that there
exists a constant $\alpha_{r,p}>0$ such that
\begin{equation}
\frac{1}{\alpha_{r,p}}C_{r,p}\left(A\right)\leq\left[c_{p,r}\left(A\right)\right]^{p}\leq\alpha_{r,p}C_{r,p}\left(A\right)\label{eq:cap_comparison}
\end{equation}
 for every $A\subset\boldsymbol{W}$. For further details about the
norms $\lVert\cdot\rVert_{r,p}$ and the corresponding capacity, one
should refer to \citep{Fukushima1993}, \citep{Watanabe1984} and
\citep{Takeda1984}.

The important properties about $(p,r)$-capacity are stated below,
which will be used in the following text. Firstly capacities $c_{p,r}$
and $C_{r,p}$ are outer measures in the sense that $c_{p,r}$ and
$C_{r,p}$ are monotonic and sub-additive, that is, $c_{p,r}(A)\leq c_{p,r}(B)$
for any $A\subseteq B$, and $c_{p,r}(A)\leq\sum_{n}c_{p,r}(A_{n})$
if $A\subset\cup_{n}A_{n}$. These properties hold for $C_{r,p}$
as well. Let us point out that the sub-additivity of $c_{p,r}$ follows
from the localization of $\lVert\cdot\rVert_{\mathbb{D}_{r}^{p}}$
, while the sub-additivity of $C_{r,p}$ follows from the triangle
inequality for norms. It follows that the first Borel-Cantelli applies
to these capacities (see e.g. Corollary 1.2.4, Chapter IV, \citep{Malliavin2015}).
More precisely, if $\{A_{n}\}_{n=1}^{\infty}$ is a sequence of subsets
of $\boldsymbol{W}$ such that $\sum_{n=1}^{\infty}c_{p,r}(A_{n})<\infty$,
then $c_{p,r}(\limsup_{n\to\infty}A_{n})=0.$ The capacity version
of the Borel-Cantelli lemma, together with the concept of the Malliavin
derivative, are the major tools in our arguments in this work. In
fact, the definition of the capacity $c_{p,r}$ implies that the following
Chebyshev's inequality (see e.g. Corollary 1.2.5, Chapter IV, \citep{Malliavin2015}).
If $\varphi\in\mathbb{D}_{r}^{p}$ and $\varphi$ is lower-semi continuous,
then
\[
c_{p,r}\left(\varphi>\lambda\right)\leq\lambda^{-1}\lVert\varphi\rVert_{\mathbb{D}_{r}^{p}}
\]
for every $\lambda>0$. 

Lemma 1.1 in \citep{Fukushima1993} with the Meyer's inequality implies
a stronger version of the sub-additivity for $c_{p,r}$, which says
that
\begin{equation}
{\color{black}{\normalcolor \left[c_{p,r}(A)\right]^{p}\leq M_{p,r}\sum_{n=1}^{\infty}\left[c_{p,r}(A_{n})\right]^{p}}}\label{eq:cap_p_subadditivity}
\end{equation}
for some constant $M_{p,r}$ depending only on $p$ and $r$, for
any $A\subset\bigcup_{n}A_{n}$. 

$c_{p,r}$ is lower continuous (see e.g. \citep{Malliavin2015}, Chapter
IV, Theorem 5.1) in the sense that for an increasing sequence of sets
$\{A_{n}\}_{n=1}^{\infty}$,
\begin{equation}
c_{p,r}\left(\bigcup_{n=1}^{\infty}A_{n}\right)=\lim_{n\to\infty}c_{p,r}(A_{n}).\label{eq:cap_lower_continuous}
\end{equation}

\subsection{Fractional Brownian motion}

In this sub-section we consider a class of Wiener functionals, fractional
Brownian motions (fBM) with Hurst parameter $H$, which are defined
as singular It\^{o}'s integrals with respect to Brownian motion.
FBMs are measurable functions on the Wiener space $(\boldsymbol{W},\mathcal{H},P)$
which are smooth in the sense of Malliavin differentiation.

An fBM $(B_{t})_{t\geq0}$ (of dimension one) with Hurst parameter
$H\in(0,1)$ is a centred Gaussian process on a probability space
$(\varOmega,\mathcal{F},\mathbb{P})$ whose covariance function is
given by
\[
R(t,s)=\mathbb{E}\left[B_{t}B_{s}\right]=\frac{1}{2}\left(t^{2H}+s^{2H}-|t-s|^{2H}\right).
\]
An fBM has stationary increments, i.e. $B_{t}-B_{s}$ and $B_{t-s}$
have the same distribution. FBMs are known as examples of self-similar
processes, i.e. for any $\alpha>0$, $\left\{ B_{t}:t\geq0\right\} =\left\{ \alpha^{-H}B_{\alpha t}:t\geq0\right\} $
in distribution. 

In this paper, fBMs will be realised as Wiener functionals on the
classical Wiener space $(\boldsymbol{W},\mathcal{H},P)$, in terms
of the following integral representation (see e.g. \citep{Decreusefond1999}):
\begin{equation}
B_{t}=\int_{0}^{t}K(t,s)d\omega(s),\label{eq:fBM_integral_repn}
\end{equation}
where the integrals on the right-hand side have to be interpreted
as It\^{o} integrals against Brownian motion $\left\{ \omega(t):t\geq0\right\} $
under the Wiener measure $P$. Here, for each pair $t>s\geq0$ define
$K$ to be the reproducing kernel 
\[
K(t,s)=\sqrt{\frac{H(2H-1)}{\beta(2-2H,H-\frac{1}{2})}}s^{\frac{1}{2}-H}\int_{s}^{t}(u-s)^{H-\frac{3}{2}}u^{H-\frac{1}{2}}du,
\]
if $H>\frac{1}{2}$, and for $H<\frac{1}{2}$, 
\[
\begin{aligned}K(t,s)= & \sqrt{\frac{2H}{(1-2H)\beta(1-2H,H+\frac{1}{2})}}\\
 & \cdot\left[\left(\frac{t}{s}\right)^{H-\frac{1}{2}}(t-s)^{H-\frac{1}{2}}-\left(H-\frac{1}{2}\right)s^{\frac{1}{2}-H}\int_{s}^{t}u^{H-\frac{3}{2}}(u-s)^{H-\frac{1}{2}}du\right],
\end{aligned}
\]
and we define $K=1$ when $H=\frac{1}{2}$, so that our results are
compatible with the classical results for Brownian motion. We notice
that $K$ is a non-negative but singular kernel and it satisfies that 

\[
\int_{0}^{t\wedge u}K(t,s)K(u,s)ds=R(t,u).
\]
For further details on the above integral representation and reproducing
kernel $K$, one may refer to \citep{Decreusefond1999} and Chapter
$5$ in \citep{Nualart2006}. $B_{t}$ (for $t>0$) are Gaussian random
variables on Wiener space $(\boldsymbol{W},\mu)$, and $\mathbb{E}\left|B_{t}-B_{s}\right|^{2}=\left|t-s\right|^{2H}$.
By choosing proper modifications of $B_{t}$ we may assume that $t\rightarrow B_{t}$
are continuous. 

For every $t\geq0$, $B_{t}$ defined by the previous integral representation
is smooth in Malliavin's sense, that is, it belongs to Sobolev space
$\mathbb{D}_{r}^{p}$ for any $r\in\mathbb{N}$ and $p\in(1,\infty)$.
In what follows, we will work with this version of fBM only. For example,
the Malliavin derivative of $B_{t}$ as a function on $\boldsymbol{W}$
can be calculated as in the following lemma, which will be used in
our main arguments.
\begin{lem}
\label{lem:Malliavin deriv}Let $H\in(0,1)$, $r\in\mathbb{N}$ and
$p\in(1,\infty)$. Then $B_{t}\in\mathbb{D}_{r}^{p}$ (for every $t>0$)
and its first order Malliavin derivative is given by 
\begin{equation}
DB_{t}(s)=\int_{0}^{s\wedge t}K(t,u)du.\label{eq:fBM_Malliavin_deriv}
\end{equation}
The higher-order derivatives of $B_{t}$ vanish (which reflects the
fact that $B_{t}$ is an integral of a deterministic function against
Brownian motion).
\end{lem}
This lemma is a Corollary to the transfer principle provided in Proposition
5.2.1, page 288, \citep{Nualart2006}. We provide an elementary proof
slightly different from that in \citep{Nualart2006} in the appendix
for completeness. 

\begin{rem}
As a consequence, according to Malliavin (Theorem 2.3.3, page 97,
\citep{Malliavin2015}), given a pair $r\in\mathbb{N}$ and $p\geq1$,
for every $\varepsilon>0$, there is an open subset $\mathcal{O}_{\varepsilon}\subseteq\boldsymbol{W}$
with $c_{p,r}(\mathcal{O}_{\varepsilon})<\varepsilon$, and there
is a family of continuous functions $\tilde{B}_{t}$ (for $t>0$)
on $\boldsymbol{W}$ such that $B_{t}=\tilde{B}_{t}$ (for all $t>0$)
$P$-a.e., and $\tilde{B}_{t}$ are continuous on $\boldsymbol{W}\setminus\mathcal{O}_{\varepsilon}$
for all $t>0$.
\end{rem}

\section{Several Technical Facts}

In this section, we shall prove several technical facts about fBM
which will be used in proving our main results. The first one is the
following inequality, which is similar to the result due to Fukushima
in \citep{Fukushima1984}, however the proof of our case is more subtle. 
\begin{lem}
\label{increment inequality}For all $H\in(0,1)$, 
\[
c_{p,r}(|B_{t}-B_{s}|>\eta)\leq\sqrt[p]{2\left[\sum_{l=0}^{r}\left(\frac{\eta}{p(t-s)^{H}}\right)^{lp}\right]}e^{-\frac{\eta^{2}}{2p(t-s)^{2H}}}
\]
 for any $r\in\mathbb{N}$, $1<p<\infty$, $\eta>0$, and $0\leq s<t$.
\end{lem}
\begin{proof}
Let $M_{s,t}=B_{t}-B_{s}$ with $0\leq s<t$. Then by the definition
of Malliavin derivative, we obtain that
\[
DM_{s,t}(u)=\int_{0}^{u}K(t,r)\mathds{1}_{[0,t]}(r)-K(s,r)\mathds{1}_{[0,s]}(r)dr\in\mathcal{H}
\]
and higher order derivatives of $M_{s,t}$ all vanish. We show that
for $\alpha\geq0$, $e^{\frac{\alpha}{p}M_{s,t}}\in\mathbb{D}_{r}^{p}$,
and 
\[
D^{l}e^{\frac{\alpha}{p}M_{s,t}}=\left(\frac{\alpha}{p}\right)^{l}e^{\frac{\alpha}{p}M_{s,t}}DM_{s,t}\otimes\cdots\otimes DM_{s,t}\in L^{p}(\boldsymbol{W};\mathcal{H}^{\otimes l})
\]
for all $1\leq l\leq r$. 

Set $f(x)=e^{\frac{\alpha}{p}x}$. For each $N\in\mathbb{N}$, let
$\psi_{N}\in C_{0}^{\infty}(\mathbb{R})$ be a cut-off function taking
values in $[0,1]$ such that 
\[
\psi_{N}(x)=\begin{cases}
1, & |x|\leq N\\
0 & |x|\geq N+1,
\end{cases}
\]
and $\sup_{x,N}|\psi_{N}^{(k)}(x)|=C<\infty$ for all $1\leq k\leq r$.
Set $f_{N}(x)=f(x)\cdot\psi_{N}(x)$. For convenience, write $F_{N}=f_{N}(M_{s,t})$,
then $F_{N}\in\mathcal{S}$ as $f_{N}\in C_{0}^{\infty}(\mathbb{R})$,
and by using the chain rule for Malliavin derivatives, we have
\[
D^{l}F_{N}=f_{N}^{(l)}(M_{s,t})DM_{s,t}\otimes\cdots\otimes DM_{s,t}
\]
for $1\leq l\leq r$. Hence,
\begin{align*}
 & \hphantom{=}\ \ \mathbb{E}\left[\Big\vert\big\Vert D^{l}F_{N}-\fracap^{l}e^{\frac{\alpha}{p}M_{s,t}}DM_{s,t}\otimes\cdots\otimes DM_{s,t}\big\Vert_{\mathcal{H}^{\otimes l}}\Big\vert^{p}\vphantom{\sum_{j=0}^{l}}\right]\\
 & =\mathbb{E}\left[\big\vert f_{N}^{(l)}(M_{s,t})-\fracap^{l}e^{\frac{\alpha}{p}M_{s,t}}\big\vert^{p}\big\Vert DM_{s,t}\big\Vert_{\mathcal{H}}^{lp}\vphantom{\sum_{j=0}^{l}}\right]\\
 & =\mathbb{E}\left[\Big\vert\sum_{j=0}^{l}\lchoosej f^{(j)}(M_{s,t})\psi_{N}^{(l-j)}(M_{s,t})-\fracap^{l}e^{\frac{\alpha}{p}M_{s,t}}\Big\vert^{p}\right]\big\Vert DM_{s,t}\big\Vert_{\mathcal{H}}^{lp}\vphantom{\left(\frac{\alpha}{p}\right)^{l}}\\
 & =\mathbb{E}\left[\Big\vert\sum_{j=0}^{l}\lchoosej f^{(j)}(M_{s,t})\psi_{N}^{(l-j)}(M_{s,t})+\fracap^{l}e^{\frac{\alpha}{p}M_{s,t}}\psi_{N}(M_{s,t})-\fracap^{l}e^{\frac{\alpha}{p}M_{s,t}}\Big\vert^{p}\right]\big\Vert DM_{s,t}\big\Vert_{\mathcal{H}}^{lp}\vphantom{\left(\frac{\alpha}{p}\right)^{l}}\\
 & \leq l^{p-1}\mathbb{E}\left[\sum_{j=0}^{l-1}\Big\vert\lchoosej f^{(j)}(M_{s,t})\psi_{N}^{(l-j)}(M_{s,t})\Big\vert^{p}+\Big\vert\fracap^{l}e^{\frac{\alpha}{p}M_{s,t}}\left(\psi_{N}(M_{s,t})-1\right)\Big\vert^{p}\right]\big\Vert DM_{s,t}\big\Vert_{\mathcal{H}}^{lp}\vphantom{\left(\frac{\alpha}{p}\right)^{l}}\\
 & \leq l^{p-1}\mathbb{E}\left[\sum_{j=0}^{l-1}\Big\vert\lchoosej\fracap^{j}e^{\frac{\alpha}{p}M_{s,t}}M\cdot\mathds{1}_{\{|M_{s,t}|\geq N\}}\Big\vert^{p}+\Big\vert\fracap^{l}e^{\frac{\alpha}{p}M_{s,t}}\mathds{1}_{\{|M_{s,t}|\geq N\}}\Big\vert^{p}\right]\big\Vert DM_{s,t}\big\Vert_{\mathcal{H}}^{lp},\vphantom{\left(\frac{\alpha}{p}\right)^{l}}
\end{align*}
which tends to zero as $N\to\infty$ by the dominated convergence
theorem. Since $F_{N}\to e^{\frac{\alpha}{p}M_{s,t}}$ as $N$ goes
to infinity in $L^{p}(\boldsymbol{W})$, and according to the previous
estimate, we get that 
\[
D^{l}F_{N}\to\left(\frac{\alpha}{p}\right)^{l}e^{\frac{\alpha}{p}M_{s,t}}DM_{s,t}\otimes\cdots\otimes DM_{s,t}
\]
 in $L^{p}(\boldsymbol{W};\mathcal{H}^{\otimes l})$. Since $D^{l}$
is closable, together with the definition of $\mathbb{D}_{l}^{p}$,
we deduce that

\[
D^{l}F=\left(\frac{\alpha}{p}\right)^{l}e^{\frac{\alpha}{p}M_{s,t}}DM_{s,t}\otimes\cdots\otimes DM_{s,t}
\]
 for each $1\leq l\leq r$ and $e^{\frac{\alpha}{p}M_{s,t}}\in\mathbb{D}_{r}^{p}$. 

By Chebyshev's inequality for $(p,r)$-capacity, it follows that
\begin{equation}
\begin{aligned}\left[c_{p,r}\left(M_{s,t}-\frac{\alpha}{2}(t-s)^{2H}>\beta\right)\right]^{p} & =\left[c_{p,r}\left(\frac{\alpha}{p}M_{s,t}-\frac{\alpha^{2}}{2p}(t-s)^{2H}>\frac{\alpha\beta}{p}\right)\right]^{p}\\
 & =\left[c_{p,r}\left(\exp\left(\frac{\alpha}{p}M_{s,t}\right)>\exp\left(\frac{\alpha^{2}}{2p}(t-s)^{2H}+\frac{\alpha\beta}{p}\right)\right)\right]^{p}\\
 & \leq\exp\left(-\frac{\alpha^{2}}{2}(t-s)^{2H}-\alpha\beta\right)\big\Vert e^{\frac{\alpha}{p}M_{s,t}}\big\Vert_{\mathbb{D}_{r}^{p}}^{p},
\end{aligned}
\label{eq:increment_ineq_chebyshev}
\end{equation}
 for any $\alpha,\beta>0$. It is clear that 
\[
\begin{aligned}\langle DM_{s,t},DM_{s,t}\rangle_{\mathcal{H}} & =\int_{0}^{\infty}\left[K(t,u)\mathds{1}_{[0,t]}(u)-K(s,u)\mathds{1}_{[0,s]}(u)\right]^{2}du\\
 & =R(t,t)-2R(s,t)+R(s,s)\\
 & =(t-s)^{2H}.
\end{aligned}
\]
Therefore,
\[
\begin{aligned}\langle D^{l}e^{\frac{\alpha}{p}M_{s,t}},D^{l}e^{\frac{\alpha}{p}M_{s,t}}\rangle_{\mathcal{H}^{\otimes l}} & =\left(\frac{\alpha}{p}\right)^{2l}e^{\frac{2\alpha}{p}M_{s,t}}(\langle DM_{s,t},DM_{s,t}\rangle_{\mathcal{H}})^{l}\\
 & =\left(\frac{\alpha}{p}\right)^{2l}e^{\frac{2\alpha}{p}M_{s,t}}(t-s)^{2lH},
\end{aligned}
\]
which implies that 
\[
\begin{aligned}\mathbb{E}\left[\left|\lVert D^{l}e^{\frac{\alpha}{p}M_{s,t}}\rVert_{\mathcal{H}^{\otimes l}}\right|^{p}\right] & =\mathbb{E}\left[\left(\frac{\alpha}{p}\right)^{lp}e^{\alpha M_{s,t}}(t-s)^{lHp}\right]\\
 & =\left(\frac{\alpha}{p}\right)^{lp}(t-s)^{lHp}e^{\frac{\alpha^{2}}{2}(t-s)^{2H}},
\end{aligned}
\]
where we have used that $M_{s,t}\sim N(0,(t-s)^{2H})$. Hence, 
\[
\begin{aligned}\big\Vert e^{\frac{\alpha}{p}M_{s,t}}\big\Vert_{\mathbb{D}_{r}^{p}}^{p} & =\mathbb{E}\left[\left\vert e^{\frac{\alpha}{p}M_{s,t}}\right\vert ^{p}\right]+\sum_{l=1}^{r}\mathbb{E}\left[\left|\lVert D^{l}e^{\frac{\alpha}{p}M_{s,t}}\rVert_{\mathcal{H}^{\otimes l}}\right|^{p}\right]\\
 & =\mathbb{E}\left[e^{\alpha M_{s,t}}\right]+\sum_{l=1}^{r}\left(\frac{\alpha}{p}\right)^{lp}(t-s)^{lHp}e^{\frac{\alpha^{2}}{2}(t-s)^{2H}}\\
 & =\left[\sum_{l=0}^{r}\left(\frac{\alpha}{p}\right)^{lp}(t-s)^{lHp}\right]e^{\frac{\alpha^{2}}{2}(t-s)^{2H}}.
\end{aligned}
\]
Now by (\ref{eq:increment_ineq_chebyshev}), we obtain that
\[
\left[c_{p,r}\left(M_{s,t}-\frac{\alpha}{2}(t-s)^{2H}>\beta\right)\right]^{p}\leq\left[\sum_{l=0}^{r}\left(\frac{\alpha}{p}\right)^{lp}(t-s)^{lHp}\right]e^{-\alpha\beta}.
\]
For any positive $\eta$, optimise the above inequality by setting
$\alpha=\frac{\eta}{(t-s)^{2H}}$ and $\beta=\frac{\eta}{2}$, and
we arrive at
\[
\begin{aligned}\left[c_{p,r}\left(M_{s,t}>\eta\right)\right]^{p} & \leq\left[\sum_{l=0}^{r}\left(\frac{\eta}{p(t-s)^{2H}}\right)^{lp}(t-s)^{lHp}\right]e^{-\frac{\eta^{2}}{2(t-s)^{2H}}}\\
 & =\sum_{l=0}^{r}\left(\frac{\eta}{p(t-s)^{H}}\right)^{lp}e^{-\frac{\eta^{2}}{2(t-s)^{2H}}}.
\end{aligned}
\]
By replacing $B$ with $-B$, we may conclude that
\[
\left[c_{p,r}\left(|M_{s,t}|>\eta\right)\right]^{p}\leq2\left[\sum_{l=0}^{r}\left(\frac{\eta}{p(t-s)^{H}}\right)^{lp}\right]e^{-\frac{\eta^{2}}{2(t-s)^{2H}}}.
\]
\end{proof}
Let $0\leq u<r<s<t\leq T$. Set 
\[
X=\frac{B_{t}-B_{s}}{(t-s)^{H}},\qquad Y=\frac{B_{r}-B_{u}}{(r-u)^{H}},
\]
so that $X,Y\sim N(0,1)$. Moreover,
\begin{equation}
\begin{aligned}\mathbb{E}[XY] & =(t-s)^{-H}(r-u)^{-H}\left(R(t,r)-R(t,u)-R(s,r)+R(s,u)\right)\\
 & =\frac{1}{2(t-s)^{H}(u-r)^{H}}\left[(t-u)^{2H}-(t-r)^{2H}-\left((s-u)^{2H}-(s-r)^{2H}\right)\right],
\end{aligned}
\label{eq:std_increment_cov}
\end{equation}
which is non-negative when $H\in\left[\frac{1}{2},1\right)$, and
non-positive when $H\in\left(0,\frac{1}{2}\right]$. We need the following
simple observation. By (\ref{eq:fBM_Malliavin_deriv}), we compute
that
\begin{equation}
\begin{aligned}\langle DX,DY\rangle_{\mathcal{H}} & =(t-s)^{-H}(r-u)^{-H}\langle DB_{t}-DB_{s},DB_{r}-DB_{u}\rangle_{\mathcal{H}}\vphantom{\int_{0}^{\infty}}\\
 & =(t-s)^{-H}(r-u)^{-H}\int_{0}^{\infty}\left(K(t,v)\mathds{1}_{[0,t]}(v)-K(s,v)\mathds{1}_{[0,s]}(v)\right)\\
 & \hphantom{=}\cdot\left(K(r,v)\mathds{1}_{[0,r]}(v)-K(u,v)\mathds{1}_{[0,u]}(v)\right)dv\vphantom{\int_{0}^{\infty}}\\
 & =\mathbb{E}[XY].\vphantom{\int_{0}^{\infty}}
\end{aligned}
\label{eq:std_increment_Hprod}
\end{equation}

Next techinical lemma contains results similar to Proposition $1$
in Fukushima \citep{Fukushima1984} and Proposition 2 in Takeda \citep{Takeda1984}.
\begin{lem}
\label{capacity-probability}For all $H\in(0,1)$ and each $N\in\mathbb{N}$,
let $0\leq t_{0}<t_{1}<\cdots<t_{N}$ with $|t_{i}-t_{i-1}|=L$, $1\leq i\leq N$.
Take $-\infty<a_{i}<b_{i}<\infty$, $c_{i}>0$, $1\leq i\leq N$.
Then it holds that 
\[
\begin{aligned}\left[c_{p,r}\left(\bigcap_{i=1}^{N}\left\{ a_{i}<X_{i}<b_{i}\right\} \right)\right]^{p}\leq & \left[\sum_{l=0}^{r}N^{lp}C_{H}^{lp/2}\left(\frac{M_{r}}{c}\right)^{lp}\right]P\left(\bigcap_{i=1}^{N}\left\{ a_{i}-c_{i}<X_{i}<b_{i}+c_{i}\right\} \right)\end{aligned}
\]
for all $r\in\mathbb{N}$ and $p\in(1,\infty)$, where 
\begin{equation}
X_{i}=\frac{B_{t_{i}}-B_{t_{i-1}}}{L^{H}}\sim N(0,1),\label{eq:defn_std_increment}
\end{equation}
$c=\min_{1\leq i\leq N}c_{i}$, $M_{r}$ is a constant depending only
on $r$, and 
\[
C_{H}=\max\left\{ 2^{2H-1}-1,1\right\} \leq1
\]
 is some constant depending only on $H$.
\end{lem}
\begin{proof}
The proof is a modification of Takeda's argument in \citep{Takeda1984}.
For $i=1,2,\cdots,N$, let $f_{i}\in C_{c}^{\infty}(\mathbb{R})$
be the cut-off functions valued in $[0,1]$ such that 
\[
f_{i}(x)=\begin{cases}
1, & x\in(a_{i},b_{i}),\\
0, & x\in(-\infty,a_{i}-c_{i})\cup(b_{i}+c_{i},\infty),
\end{cases}
\]
and 
\[
\left|\frac{d^{l}f_{i}}{dx^{l}}\right|\leq\frac{M_{r}}{c_{i}^{l}}
\]
for all $l\leq r$, where $M_{r}\geq1$ is a constant depending on
$r$. Set $F(x_{1,}\cdots,x_{N})=\prod_{i=1}^{N}f_{i}(x_{i}),$ then
according to the above conditions, we have that
\begin{equation}
\left|\partial_{n_{1},\cdots,n_{l}}^{l}F(x_{1,}\cdots,x_{N})\right|\leq\left(\frac{M_{r}}{c}\right)^{l}\mathds{1}_{\prod_{i=1}^{N}(a_{i}-c_{i},b_{i}+c_{i})}(x_{1},\cdots,x_{N})\label{eq:cut_off_deriv_con}
\end{equation}
 for each $l\leq r$, where $c=\min_{1\leq i\leq N}c_{i}$. For simplicity,
write $Y=F(X_{1},\cdots,X_{N})$, where $X_{i}$'s are defined as
in (\ref{eq:defn_std_increment}). Then $Y\in\mathbb{D}_{l}^{p}$,
and since all Malliavin derivatives of $X_{i}$ with order higher
than $2$ vanish, it holds that
\[
D^{l}Y=\sum_{\mathsmaller{1\leq n_{1},\cdots,n_{l}\leq N}}\partial_{n_{1},\cdots,n_{l}}^{l}F(X_{1},\cdots,X_{N})DX_{n_{1}}\otimes\cdots\otimes DX_{n_{l}}.
\]
Moreover, $D^{l}F\in\mathcal{H}^{\otimes l}$ and
\begin{equation}
\begin{aligned}\lVert D^{l}Y\rVert_{\mathcal{H}^{\otimes l}}^{2}= & \sum_{\substack{\mathsmaller{1\leq n_{1},\cdots,n_{l}\leq N}\\
\mathsmaller{1\leq m_{1},\cdots,m_{l}\leq N}
}
}\left(\partial_{n_{1},\cdots,n_{l}}^{l}F(X_{1},\cdots,X_{N})\partial_{m_{1},\cdots,m_{l}}^{l}F(X_{1},\cdots,X_{N})\vphantom{\prod_{i=1}^{l}}\right.\\
 & \left.\hphantom{\sum_{\substack{1\leq n_{1},\cdots,n_{l}\leq N\\
1\leq m_{1},\cdots,m_{l}\leq N
}
}}\cdot\prod_{i=1}^{l}\langle DX_{n_{i}},DX_{m_{i}}\rangle_{\mathcal{H}}\right).
\end{aligned}
\label{eq:cap_prob_Hnorm}
\end{equation}

Our next step is to find an upper bound for $|\langle DX_{j},DX_{k}\rangle_{\mathcal{H}}|$
for all $1\leq j,k\leq N$. When $1\leq j=k\leq N$, $\langle DX_{j},DX_{k}\rangle_{\mathcal{H}}=1$;
when $1\leq j<k\leq N$, by (\ref{eq:std_increment_cov}) and (\ref{eq:std_increment_Hprod}),
\[
\begin{aligned}\langle DX_{i},DX_{j}\rangle_{\mathcal{H}} & =\mathbb{E}[X_{j}X_{k}]\\
 & =\frac{1}{2}\left[\left(k-j+1\right)^{2H}+\left(k-j-1\right)^{2H}-2\left(k-j\right)^{2H}\right].
\end{aligned}
\]

Set $g(x)=\frac{1}{2}\left[(x+1)^{2H}+(x-1)^{2H}-2x^{2H}\right]$.
Observe that when $H<\frac{1}{2}$, $x^{2H}$ is concave, so $g(x)\leq0$,
and similarly when $H>\frac{1}{2}$, $g(x)\geq0$. The derivative
of $g$ is given by 
\[
g'(x)=H\left[\left((x+1)^{2H-1}-x^{2H-1}\right)-\left(x^{2H-1}-(x-1)^{2H-1}\right)\right].
\]
Using the fact that the function $x^{2H-1}$ is convex if $H\in\left(0,\frac{1}{2}\right)$,
we deduce that when $H\in\left(0,\frac{1}{2}\right)$, $g'(x)\geq0$.
As $k-j\in\{1,2,\cdots,N-1\}$, it follows that
\[
\left\vert \langle DX_{j},DX_{k}\rangle_{\mathcal{H}}\right\vert \leq2^{2H-1}-1.
\]
When $H\in\left(\frac{1}{2},1\right)$, $g'(x)\leq0$ and thus $|\langle DX_{j},DX_{k}\rangle_{\mathcal{H}}|\leq2^{2H-1}-1$
. Set $C_{H}=\max\left\{ 2^{2H-1}-1,1\right\} $, then $|\langle DX_{j},DX_{k}\rangle_{\mathcal{H}}|\leq C_{H}$
for all $1\leq j,k\leq N$. Moreover, as $H$ takes values in $(0,1)$,
$C_{H}\leq1$.

Therefore, by (\ref{eq:cap_prob_Hnorm}), together with (\ref{eq:cut_off_deriv_con}),
it follows that
\[
\lVert D^{l}Y\rVert_{\mathcal{H}^{\otimes l}}^{2}\leq N^{2l}\left(\frac{M_{r}}{c}\right)^{2l}\mathds{1}_{\prod_{i=1}^{N}(a_{i}-c_{i},b_{i}+c_{i})}(X_{1},\cdots,X_{N})C_{H}^{l}
\]
for all $l\leq r$. Hence
\[
\left|\lVert D^{l}Y\rVert_{\mathcal{H}^{\otimes l}}\right|^{p}\leq N^{lp}C_{H}^{lp/2}\left(\frac{M_{r}}{c}\right)^{lp}\mathds{1}_{\prod_{i=1}^{N}(a_{i}-c_{i},b_{i}+c_{i})}(X_{1},\cdots,X_{N}).
\]
By the definition of $(p,r)$-capacity,
\begin{align*}
 & \left[c_{p,r}\left(\bigcap_{i=1}^{N}\left\{ a_{i}<X_{i}<b_{i}\right\} \right)\right]^{p}\\
\leq & \vphantom{\biggr\updownarrow}\hphantom{(}\lVert Y\rVert_{\mathbb{D}_{r}^{p}}^{p}\\
= & \vphantom{\vphantom{\biggr\updownarrow}}\hphantom{(}\mathbb{E}\left[\lvert Y\rvert^{p}\right]+\sum_{l=1}^{r}\mathbb{E}\left[\left|\lVert D^{l}Y\rVert_{\mathcal{H}^{\otimes l}}\right|^{p}\right]\\
\leq & \hphantom{(}P\left(\bigcap_{i=1}^{N}\left\{ a_{i}-c_{i}<X_{i}<b_{i}+c_{i}\right\} \right)\\
 & +\sum_{l=1}^{r}\left(N^{lp}C_{H}^{lp/2}\left(\frac{M_{r}}{c}\right)^{lp}\right)P\left(\bigcap_{i=1}^{N}\left\{ a_{i}-c_{i}<X_{i}<b_{i}+c_{i}\right\} \right)\\
= & \left[\sum_{l=0}^{r}N^{lp}C_{H}^{lp/2}\left(\frac{M_{r}}{c}\right)^{lp}\right]P\left(\bigcap_{i=1}^{N}\left\{ a_{i}-c_{i}<X_{i}<b_{i}+c_{i}\right\} \right).
\end{align*}
\end{proof}
Throughout this paper, we always use the notation $X_{\cdot}$ to
denote normalised increment of fBM, though it may refer to increment
over time interval of different length, it always has standard Gaussian
distribution.

The third technical lemma we need is a $(2,1)$-capacity estimate
on the supremum process for fBM with Hurst parameter $H\in\left(0,\frac{1}{2}\right)$,
whose proof is quite technical due to lack of suitable tools such
as Doob's maximal inequality for martingales. We overcome the difficulties
by carefully applying Slepian's lemma for related Gaussian processes. 
\begin{lem}
\label{lem:supremum inequality} Let $0\leq s<t$. For $H\in(0,1)$
and $\eta>0$,
\begin{equation}
c_{2,1}\left(\sup_{s\leq u\leq t}\left(B_{u}-B_{s}\right)>\eta\right)\leq C_{s,t,\eta,H}\cdot\exp\left(-\frac{\eta^{2}}{4\left[\gamma_{H}(t-s)^{2H}+(t-s)\right]}\right),\label{eq:sup_ineq_1}
\end{equation}
and 
\begin{equation}
c_{2,1}\left(\sup_{s\leq u\leq t}\left\vert B_{u}-B_{s}\right\vert >\eta\right)\leq\sqrt{2}C_{s,t,\eta,H}\cdot\exp\left(-\frac{\eta^{2}}{4\left[\gamma_{H}(t-s)^{2H}+(t-s)\right]}\right),\label{eq:sup_ineq_2}
\end{equation}
\begin{equation}
c_{2,1}\left(\sup_{s\leq u\leq t}\left\vert B_{t}-B_{u}\right\vert >\eta\right)\leq\sqrt{2}C_{s,t,\eta,H}\cdot\exp\left(-\frac{\eta^{2}}{4\left[\gamma_{H}(t-s)^{2H}+(t-s)\right]}\right),\label{eq:sup_ineq_3}
\end{equation}
where 
\begin{equation}
\gamma_{H}=\begin{cases}
1, & H\leq\frac{1}{2},\\
\frac{3}{2}, & H>\frac{1}{2},
\end{cases}\label{eq:gamma_H}
\end{equation}
and 
\[
C_{s,t,\eta,H}=\sqrt{\frac{\eta^{2}(t-s)^{2H}}{2\left[\gamma_{H}(t-s)^{2H}+(t-s)\right]^{2}}+2}.
\]
\end{lem}
\begin{proof}
We shall follow the same ideas as for the proof of Proposition 2 and
3 in \citep{Fukushima1984}, while we have to overcome several difficulties
arising from the fact that the distribution of supremum process is
not known for fBM. When $H=\frac{1}{2}$, the above inequality is
covered by the result due to Fukushima in \citep{Fukushima1984}.

We prove (\ref{eq:sup_ineq_1}) and (\ref{eq:sup_ineq_2}) first.
For simplicity, define $M_{s,t}^{*}=\sup_{s\leq u\leq t}\left(B_{u}-B_{s}\right)$
for any $0\leq s<t$. Following Fukushima's notation in \citep{Fukushima1984},
for $s<t_{1}<\cdots<t_{n}\leq t,$ let us define $B_{s;t_{1},\cdots,t_{n}}=\left(B_{t_{1}}-B_{s},\cdots,B_{t_{n}}-B_{s}\right)$,
and let $g(x_{1},\cdots,x_{n})=x_{1}\vee\cdots\vee x_{n}$, and define
\[
M_{s;t_{1},\cdots,t_{n}}=g(B_{s;t_{1},\cdots,t_{n}})=\max_{1\leq i\leq n}\left(B_{t_{i}}-B_{s}\right).
\]

We proceed in 4 steps. 

\textbf{Step 1.} In this step, only the law of fBM will be involved,so
the argument is applicable to various Gaussian processes. As $t_{i}$'s
are fixed in the first two steps, we simplify our notations by writing
$B_{s,t}^{(n)}=B_{s;t_{1},\cdots,t_{n}}$ and $M_{s,t}^{(n)}=M_{s;t_{1},\cdots,t_{n}}$
for the moment. In this step, we establish an upper bound for $\mathbb{E}\left[e^{\alpha M_{s,t}^{(n)}}\right]$,
where $\alpha>0$. 

Consider the following correlation: 
\begin{equation}
\mathbb{E}\left[\left(B_{t_{i}}-B_{s}\right)\left(B_{t_{j}}-B_{s}\right)\right]=\mathbb{E}\left[\left(B_{t_{i}}-B_{s}\right)^{2}+\left(B_{t_{i}}-B_{s}\right)\left(B_{t_{j}}-B_{t_{i}}\right)\right].\label{eq:Correlation_increment}
\end{equation}

When $H<\frac{1}{2}$, for any $1\leq i\leq j\leq n$, the increments
of $\left(B_{t}\right)_{t\geq0}$ over different time intervals are
negatively correlated, which leads to
\begin{equation}
\begin{aligned}\mathbb{E}\left[\left(B_{t_{i}}-B_{s}\right)\left(B_{t_{j}}-B_{s}\right)\right] & \leq\mathbb{E}\left[\left(B_{t_{i}}-B_{s}\right)^{2}\right]\\
 & =(t_{i}-s)^{2H}\\
 & \leq(t-s)^{2H}.
\end{aligned}
\label{eq:Correlation_bound_H<1/2}
\end{equation}

When $H>\frac{1}{2}$, we seek for an upper bound of 
\[
\mathbb{E}\left[\left(B_{t_{i}}-B_{s}\right)\left(B_{t_{j}}-B_{t_{i}}\right)\right].
\]
We compute that
\[
\mathbb{E}\left[\left(B_{t_{i}}-B_{s}\right)\left(B_{t_{j}}-B_{t_{i}}\right)\right]=\frac{1}{2}\left[(t_{j}-s)^{2H}-(t_{j}-t_{i})^{2H}-(t_{i}-s)^{2H}\right]\leq\frac{1}{2}(t-s)^{2H},
\]
where $0\leq s<t_{i}<t_{j}\leq t$. Combining with (\ref{eq:Correlation_increment}),
we have 
\[
\mathbb{E}\left[\left(B_{t_{i}}-B_{s}\right)\left(B_{t_{j}}-B_{s}\right)\right]\leq\frac{3}{2}(t-s)^{2H}.
\]

Therefore, for all $H\in(0,1)$, 
\[
\mathbb{E}\left[\left(B_{t_{i}}-B_{s}\right)\left(B_{t_{j}}-B_{s}\right)\right]\leq\gamma_{H}(t-s)^{2H},
\]
where $\gamma_{H}$ is defined as in (\ref{eq:gamma_H}).

For convenience, set $Z_{i}=B_{t_{i}}-B_{s}\sim N(0,(t_{i}-s)^{2H})$,
and by the above estimate, correlations between any two $Z_{i}$'s
are bounded by $\gamma_{H}(t-s)^{2H}$. We want to apply Slepian's
lemma (see \citep{Ledoux2013}) to overcome the difficulties in finding
the distribution of supremum process of fBM, so we take a random variable
$\xi_{s,t}\sim N(0,\gamma_{H}(t-s)^{2H})$ independent of the standard
Brownian motion $(\omega_{t})_{t\geq0}$ on $(\boldsymbol{W},\mathcal{H},P)$,
so that $\xi_{s,t}$ and $\omega_{t_{i}}-\omega_{s}$ are independent
for all $i\in\{1,2,\cdots,n\}$. Define $Y_{i}=\omega_{t_{i}}-\omega_{s}+\xi_{s,t}$,
$1\leq i\leq n$ and let 
\[
N_{s,t}^{(n)}=\max_{1\leq i\leq n}Y_{i}=\max_{1\leq i\leq n}\left(\omega_{t_{i}}-\omega_{s}\right)+\xi_{s,t}.
\]
Then by independence, 
\[
\begin{aligned}\mathbb{E}\left[Y_{i}Y_{j}\right] & =\mathbb{E}\left[(\omega_{t_{i}}-\omega_{s}+\xi_{s,t})(\omega_{t_{j}}-\omega_{s}+\xi_{s,t})\right]\\
 & =\mathbb{E}\left[(\omega_{t_{i}}-\omega_{s})(\omega_{t_{j}}-\omega_{s})\right]+\mathbb{E}\left[\xi_{s,t}^{2}\right]\\
 & =t_{i}-s+\gamma_{H}(t-s)^{2H}\\
 & \geq\gamma_{H}(t-s)^{2H},
\end{aligned}
\]
for $1\leq i\leq j\leq n$, and hence by (\ref{eq:Correlation_bound_H<1/2}),
\[
\mathbb{E}[Z_{i}Z_{j}]\leq\mathbb{E}[Y_{i}Y_{j}]
\]
for all $1\leq i,j\leq n$. Since both exponential function and maximum
function are convex, their composition is also convex, and hence according
to Theorem 3.11 in Ledoux and Talagrand \citep{Ledoux2013}, Slepian's
lemma, we obtain that 
\[
\mathbb{E}\left[e^{\alpha M_{s,t}^{(n)}}\right]=\mathbb{E}\left[e^{\alpha\max_{1\leq i\leq n}Z_{i}}\right]\leq\mathbb{E}\left[e^{\alpha\max_{1\leq i\leq n}Y_{i}}\right]=\mathbb{E}\left[e^{\alpha N_{s,t}^{(n)}}\right],
\]
for all $\alpha>0$. Due to independence and the fact that $\max_{1\leq i\leq n}\left(\omega_{t_{i}}-\omega_{s}\right)\leq\sup_{s\leq u\leq t}\left(\omega_{u}-\omega_{s}\right)$,
\begin{align*}
\mathbb{E}\left[e^{\alpha N_{s,t}^{(n)}}\right] & =\mathbb{E}\left[e^{\alpha\xi_{s,t}}\right]\mathbb{E}\left[\exp\left(\alpha\max_{1\leq i\leq n}\left(\omega_{t_{i}}-\omega_{s}\right)\right)\right]\\
 & \leq\exp\left(\frac{\alpha^{2}}{2}\gamma_{H}(t-s)^{2H}\right)\mathbb{E}\left[\exp\left(\alpha\sup_{s\leq u\leq t}\left(\omega_{u}-\omega_{s}\right)\right)\right].
\end{align*}
Using the distribution of supremum of standard Brownian motion, we
obtain that 
\begin{equation}
\mathbb{E}\left[e^{\alpha M_{s,t}^{(n)}}\right]=\mathbb{E}\left[\exp\left(\alpha M_{s;t_{1},\cdots,t_{n}}\right)\right]\leq2\exp\left(\frac{\alpha^{2}}{2}\left[\gamma_{H}(t-s)^{2H}+(t-s)\right]\right).\label{eq:mgf_sup_BM}
\end{equation}

\textbf{Step 2. }The difference from classical approach will be demonstrated
in this step since we use only the Brownian motion capacity. In this
step, we show that $e^{\frac{\alpha}{2}M_{s,t}^{(n)}}\in\mathbb{D}_{1}^{2}$
and 
\[
\begin{aligned}De^{\frac{\alpha}{2}M_{s,t}^{(n)}} & =\frac{\alpha}{2}\exp\left(\frac{\alpha}{2}M_{s;t_{1},\cdots,t_{n}}\right)DM_{s,t}^{(n)}.\end{aligned}
\]

Observe that $g$ is Lipschitz, so by Proposition 1.2.4 in Nualart
\citep{Nualart2006}, $M_{s,t}^{(n)}=g(B_{s,t}^{(n)})\in\mathbb{D}_{1}^{2}$,
and the chain rule applies, which is
\[
\begin{aligned}DM_{s,t}^{(n)}(u) & =\sum_{i=1}^{n}\mathds{1}_{\{M_{s,t}^{(n)}=B_{t_{i}}-B_{s}\}}(B_{s,t}^{(n)})D(B_{t_{i}}-B_{s})\\
 & =\sum_{i=1}^{n}\mathds{1}_{\{M_{s,t}^{(n)}=B_{t_{i}}-B_{s}\}}(B_{s,t}^{(n)})\left[K(t_{i},u)\mathds{1}_{[0,t_{i}]}(u)-K(s,u)\mathds{1}_{[0,s]}(u)\right].
\end{aligned}
\]
Therefore, we have
\begin{align}
\langle DM_{s,t}^{(n)},DM_{s,t}^{(n)}\rangle_{\mathcal{H}} & =\sum_{i=1}^{n}\mathds{1}_{\{M_{s,t}^{(n)}=B_{t_{i}}-B_{s}\}}(B_{s,t}^{(n)})\nonumber \\
 & \hphantom{=}\cdot\int_{0}^{\infty}\left[K(t_{i},u)\mathds{1}_{[0,t_{i}]}(u)-K(s,u)\mathds{1}_{[0,s]}(u)\right]^{2}du\nonumber \\
 & =\sum_{i=1}^{n}\mathds{1}_{\{M_{s,t}^{(n)}=B_{t_{i}}-B_{s}\}}(B_{s,t}^{(n)})(t_{i}-s)^{2H}\nonumber \\
 & \leq(t-s)^{2H}\sum_{i=1}^{n}\mathds{1}_{\{M_{s,t}^{(n)}=B_{t_{i}}-B_{s}\}}.\label{eq:max_Hnorm}
\end{align}
Similar to the argument in lemma \ref{increment inequality}, we set
$f(x)=e^{\frac{\alpha}{2}x}$, and $\psi_{N}(x)$ as in lemma \ref{increment inequality},
and set $f_{N}=f\cdot\psi_{N}$. For simplicity, denote $F=e^{\frac{\alpha}{2}M_{s,t}^{(n)}}$,
and $F_{N}=f_{N}(M_{s,t}^{(n)})$. Then since $f_{N}\in C_{0}^{\infty}(\mathbb{R})$,
the chain rule applies, and 
\[
DF_{N}=f_{N}'(M_{s,t}^{(n)})DM_{s,t}^{(n)}.
\]
Similarly, we have that 
\begin{align*}
 & \hphantom{=\ \ }\mathbb{E}\left[\big\Vert DF_{N}-\frac{\alpha}{2}e^{\frac{\alpha}{2}M_{s,t}^{(n)}}DM_{s,t}^{(n)}\big\Vert_{\mathcal{H}}^{2}\right]\\
 & =\int_{X}\left|f'(M_{s,t}^{(n)})\psi_{N}(M_{s,t}^{(n)})+f(M_{s,t}^{(n)})\psi_{N}'(M_{s,t}^{(n)})-\frac{\alpha}{2}e^{\frac{\alpha}{2}M_{s,t}^{(n)}}\right|^{2}\lVert DM_{s,t}^{(n)}\rVert_{\mathcal{H}}^{2}dP\\
 & \leq\int_{X}\left|\frac{\alpha}{2}e^{\frac{\alpha}{2}M_{s,t}^{(n)}}\left(\psi_{N}(M_{s,t}^{(n)})-1\right)+e^{\frac{\alpha}{2}M_{s,t}^{(n)}}\psi_{N}'(M_{s,t}^{(n)})\right|^{2}\left(\sum_{i=1}^{n}\mathds{1}_{\{M_{s,t}^{(n)}=B_{t_{i}}-B_{s}\}}\right)(t-s)^{2H}dP\\
 & \leq2\mathbb{E}\left[\left|\frac{\alpha}{2}e^{\frac{\alpha}{2}M_{s,t}^{(n)}}\cdot\mathds{1}_{\{|M_{s,t}^{(n)}|\geq N\}}\right|^{2}+\left|e^{\frac{\alpha}{2}M_{s,t}^{(n)}}C\cdot\mathds{1}_{\{|M_{s,t}^{(n)}|\geq N\}}\right|^{2}\right](t-s)^{2H},
\end{align*}
which tends to zero as $N\to\infty$, where $C$ is defined as in
lemma \ref{increment inequality}. Therefore, since $F_{N}\to F$
in $L^{2}(\boldsymbol{W})$, $DF_{N}\to\frac{\alpha}{2}e^{\frac{\alpha}{2}M_{s,t}^{(n)}}DM_{s,t}^{(n)}$
in $L^{2}(\boldsymbol{W};\mathcal{H})$ and $D$ is closable from
$L^{2}(\boldsymbol{W})$ to $L^{2}(\boldsymbol{W};\mathcal{H})$,
it follows that 
\begin{equation}
DF=\frac{\alpha}{2}e^{\frac{\alpha}{2}M_{s,t}^{(n)}}DM_{s,t}^{(n)}\label{eq: Dexp(aM^n_s,t)}
\end{equation}
 and $F\in\mathbb{D}_{1}^{2}$.

\textbf{Step 3.} In this step, we find an upper bound for $\mathbb{E}\left[e^{\alpha M_{s,t}^{*}}\right]$
for any $\alpha>0$, then we prove that $e^{\alpha M_{s,t}^{*}}\in\mathbb{D}_{1}^{2}$
and find an upper bound for $\Vert e^{\alpha M_{s,t}^{*}}\Vert_{\mathbb{D}_{1}^{2}}$.
As $M_{s;t_{1},\cdots,t_{n}}$ increases to $M_{s,t}^{*}$ when we
refine the partition and let $n$ go to infinity, the monotone convergence
theorem and (\ref{eq:mgf_sup_BM}) implies that 
\[
\mathbb{E}\left[e^{\alpha M_{s,t}^{*}}\right]\leq2\exp\left(\frac{\alpha^{2}}{2}\left[\gamma_{H}(t-s)^{2H}+(t-s)\right]\right).
\]

We have already proved that $e^{\alpha M_{s;t_{1},\cdots,t_{n}}}\in\mathbb{D}_{1}^{2}$
in last step, by (\ref{eq:max_Hnorm}) and (\ref{eq: Dexp(aM^n_s,t)}),
\begin{align*}
\langle D\exp\left(\frac{\alpha}{2}M_{s;t_{1},\cdots,t_{n}}\right),D\exp\left(\frac{\alpha}{2}M_{s;t_{1},\cdots,t_{n}}\right)\rangle_{\mathcal{H}} & \leq\frac{\alpha^{2}}{4}(t-s)^{2H}\exp\left(\alpha M_{s;t_{1},\cdots,t_{n}}\right)\\
 & \hphantom{=}\cdot\sum_{i=1}^{n}\mathds{1}_{\{M_{s;t_{1},\cdots,t_{n}}=B_{t_{i}}-B_{s}\}}.
\end{align*}
Therefore, by (\ref{eq:mgf_sup_BM}), we obtain that
\[
\begin{aligned} & \hphantom{\leq\ \ }\mathbb{E}\left[\langle D\exp\left(\frac{\alpha}{2}M_{s;t_{1},\cdots,t_{n}}\right),D\exp\left(\frac{\alpha}{2}M_{s;t_{1},\cdots,t_{n}}\right)\rangle_{\mathcal{H}}\right]\\
 & \leq\frac{\alpha^{2}}{4}(t-s)^{2H}\sum_{i=1}^{n}\int_{\{M_{s;t_{1},\cdots,t_{n}}=B_{t_{i}}-B_{s}\}}\exp\left(\alpha M_{s;t_{1},\cdots,t_{n}}\right)dP\\
 & =\frac{\alpha^{2}}{4}(t-s)^{2H}\mathbb{E}\left[\exp\left(\alpha M_{s;t_{1},\cdots,t_{n}}\right)\right]\\
 & \leq\frac{\alpha^{2}}{2}(t-s)^{2H}\exp\left(\frac{\alpha^{2}}{2}\left[\gamma_{H}(t-s)^{2H}+(t-s)\right]\right),
\end{aligned}
\]
which implies that
\[
\sup_{n}\mathbb{E}\left[\lVert D\exp\left(\frac{\alpha}{2}M_{s;t_{1},\cdots,t_{n}}\right)\rVert_{\mathcal{H}}^{2}\right]<\infty.
\]
Applying Lemma 1.2.3 in \citep{Nualart2006}, we deduce that $e^{\frac{\alpha}{2}M_{s,t}^{*}}\in\mathbb{D}_{1}^{2}$
and 
\[
D\exp\left(\frac{\alpha}{2}M_{s;t_{1},\cdots,t_{n}}\right)\to De^{\frac{\alpha}{2}M_{s,t}^{*}}
\]
 weakly in $L^{2}(\boldsymbol{W};\mathcal{H})$. As a consequence,
we have
\[
\begin{aligned}\mathbb{E}\left[\langle De^{\frac{\alpha}{2}M_{s,t}^{*}},De^{\frac{\alpha}{2}M_{s,t}^{*}}\rangle_{\mathcal{H}}\right] & \leq\liminf_{n\to\infty}\mathbb{E}\left[\langle D\exp\left(\frac{\alpha}{2}M_{s;t_{1},\cdots,t_{n}}\right),D\exp\left(\frac{\alpha}{2}M_{s;t_{1},\cdots,t_{n}}\right)\rangle_{\mathcal{H}}\right]\\
 & \leq\frac{\alpha^{2}}{4}(t-s)^{2H}\mathbb{E}\left[e^{\alpha M_{s,t}^{*}}\right]\\
 & \leq\frac{\alpha^{2}}{2}(t-s)^{2H}\exp\left(\frac{\alpha^{2}}{2}\left[\gamma_{H}(t-s)^{2H}+(t-s)\right]\right).
\end{aligned}
\]
Therefore,
\begin{equation}
\begin{aligned}\big\Vert e^{\frac{\alpha}{2}M_{s,t}^{*}}\big\Vert_{\mathbb{D}_{1}^{2}}^{2} & =\mathbb{E}\left[e^{\alpha M_{s,t}^{*}}\right]+\mathbb{E}\left[\langle De^{\frac{\alpha}{2}M_{s,t}^{*}},De^{\frac{\alpha}{2}M_{s,t}^{*}}\rangle_{\mathcal{H}}\right]\\
 & \leq\left(\frac{\alpha^{2}}{2}(t-s)^{2H}+2\right)\exp\left(\frac{\alpha^{2}}{2}\left[\gamma_{H}(t-s)^{2H}+(t-s)\right]\right).
\end{aligned}
\label{eq:M*_1,2_norm}
\end{equation}

\textbf{Step 4. }By Chebyshev's inequality for capacity and (\ref{eq:M*_1,2_norm}),
we thus have
\begin{equation}
\begin{aligned} & \hphantom{=\ \ }\left[c_{2,1}\left(M_{s,t}^{*}-\frac{\alpha}{2}(t-s)^{2H}>\beta\right)\right]^{2}\\
 & =\left[c_{2,1}\left(\frac{\alpha}{2}M_{s,t}^{*}-\frac{\alpha^{2}}{4}(t-s)^{2H}>\frac{\alpha\beta}{2}\right)\right]^{2}\\
 & =\left[c_{2,1}\left(\exp\left(\frac{\alpha}{2}M_{s,t}^{*}\right)>\exp\left(\frac{\alpha\beta}{2}+\frac{\alpha^{2}}{4}(t-s)^{2H}\right)\right)\right]^{2}\\
 & \leq\exp\left(-\alpha\beta-\frac{\alpha^{2}}{2}(t-s)^{2H}\right)\big\Vert e^{\frac{\alpha}{2}M_{s,t}^{*}}\big\Vert_{\mathbb{D}_{1}^{2}}^{2}\\
 & \leq\left(\frac{\alpha^{2}}{2}(t-s)^{2H}+2\right)\exp\left(-\alpha\beta+\frac{\alpha^{2}}{2}\left[(\gamma_{H}-1)(t-s)^{2H}+(t-s)\right]\right)
\end{aligned}
\label{eq:sup_chebyshev}
\end{equation}
for any positive constants $\alpha$ and $\beta$. 

Notice that the exponential function is the dominating part in the
last term of (\ref{eq:sup_chebyshev}), so we optimise the above quantity
by minimising the exponent and setting 
\[
\alpha=\frac{\eta}{\gamma_{H}(t-s)^{2H}+(t-s)},
\]
and 
\[
\beta=\eta-\frac{\alpha}{2}(t-s)^{2H}.
\]
Therefore, we get that
\[
\left[c_{2,1}\left(M_{s,t}^{*}>\eta\right)\right]^{2}\leq C_{s,t,\eta,H}^{2}\exp\left(-\frac{\eta^{2}}{2\left[\gamma_{H}(t-s)^{2H}+(t-s)\right]}\right),
\]
where 
\[
C_{s,t,\eta,H}=\sqrt{\frac{\eta^{2}(t-s)^{2H}}{2\left[\gamma_{H}(t-s)^{2H}+(t-s)\right]^{2}}+2}.
\]
Moreover, by replacing $B$ with $-B$, it follows that
\[
\left[c_{2,1}\left(\sup_{s\leq u\leq t}\left|B_{u}-B_{s}\right|>\eta\right)\right]^{2}\leq2C_{s,t,\eta,H}^{2}\exp\left(-\frac{\eta^{2}}{2\left[\gamma_{H}(t-s)^{2H}+(t-s)\right]}\right).
\]

Finally, (\ref{eq:sup_ineq_3}) may be established directly following
the same argument with slight modification in the definition of $M_{s,t}^{*}$.
\end{proof}
\begin{rem}
The results in the previous lemma can be considered as the maximal
inequality for fBMs but with respect to Brownian motion capacity.
For a similar result when $H=\frac{1}{2}$, one may refer to Fukushima
\citep{Fukushima1984}, or Takeda \citep{Takeda1984} for any $r\in\mathbb{N}$
and $p\in(1,\infty)$. Though we establish the inequalities for all
$H\in(0,1)$, when considering a sufficiently small time interval
$[s,t]$, the result looks weaker when $H>\frac{1}{2}$ due to the
appearance of $(t-s)$ in the exponent. In fact, when $H>\frac{1}{2}$,
$(t-s)$ will be the dominating part rather than $(t-s)^{2H}$. However,
the factor $(t-s)$ appears necessary for small time intervals.
\end{rem}

\section{Modulus of Continuity}

In this part, we shall show the result on modulus of continuity for
fractional Brownian motion with respect to the $(p,r)$-capacity defined
on classical Wiener space. We shall adopt the arguments in Fukushima's
work \citep{Fukushima1984} and the original proof by L\'{e}vy \citep{Levy1937},
who proved the modulus of continuity of Brownian motion in probability
sense. 
\begin{thm}
Let $(B_{t})_{t\geq0}$ be an fBM with Hurst parameter $H$. Then
it holds that 
\begin{equation}
\limsup_{\delta\downarrow0}\frac{1}{\sqrt{2\delta^{2H}\log(1/\delta)}}\max_{\substack{\mathsmaller{0\leq s<t\leq1}\\
\mathsmaller{t-s\leq\delta}
}
}\lvert B_{t}-B_{s}\rvert\leq1,\quad\text{q.s.}\label{eq:continuity_upp_bound}
\end{equation}
when $H\in(0,1)$ and 
\begin{equation}
\limsup_{\delta\downarrow0}\frac{1}{\sqrt{2\delta^{2H}\log(1/\delta)}}\max_{\substack{\mathsmaller{0\leq s<t\leq1}\\
\mathsmaller{t-s\leq\delta}
}
}\lvert B_{t}-B_{s}\rvert\geq1,\quad\text{q.s.}\label{eq:continuity_low_bound}
\end{equation}
when $H\in(0,\frac{1}{2}]$.
\end{thm}
\begin{proof}
Let us prove (\ref{eq:continuity_low_bound}) first. For any $r\in\mathbb{N}$
and $p\in(1,\infty)$, we want to show that 
\[
c_{p,r}\left(\limsup_{\delta\downarrow0}\frac{1}{g(\delta)}\max_{\substack{\mathsmaller{0\leq s<t\leq1}\\
\mathsmaller{t-s\leq\delta}
}
}\lvert B_{t}-B_{s}\rvert<1\right)=0,
\]
where $g(\delta)=\sqrt{2\delta^{2H}\log(1/\delta)}$. 

By lemma \ref{capacity-probability}, we have
\begin{align*}
 & \hphantom{=\ \ }c_{p,r}\left(\max_{1\leq j\leq2^{n}}\left\vert B_{\frac{j}{2^{n}}}-B_{\frac{j-1}{2^{n}}}\right\vert \leq(1-\theta)g(2^{-n})\right)\\
 & =c_{p,r}\left(\max_{1\leq j\leq2^{n}}2^{nH}\left\vert B_{\frac{j}{2^{n}}}-B_{\frac{j-1}{2^{n}}}\rvert\right\vert \leq(1-\theta)\sqrt{2\log(2^{n})}\right)\\
 & \leq\left[\sum_{k=0}^{r}N^{kp}C_{H}^{kp/2}\left(\frac{M_{r}}{c}\right)^{kp}\right]^{1/p}\left[P\left(\max_{1\leq j\leq2^{n}}2^{nH}\left|B_{\frac{j}{2^{n}}}-B_{\frac{j-1}{2^{n}}}\right|\leq(1-\theta)\sqrt{2\log(2^{n})}+c\right)\right]^{1/p}
\end{align*}
 for $\theta\in(0,1)$, where $c$ is some small constant such that
$c<\theta\sqrt{2\log2}$. Set
\[
X_{j}=2^{nH}\left(B_{\frac{j}{2^{n}}}-B_{\frac{j-1}{2^{n}}}\right)\sim N(0,1),
\]
then $X_{j}\sim N(0,1)$ and when $H\leq\frac{1}{2}$, $\mathbb{E}[X_{j}X_{k}]\leq0$
for $j\neq k$. Take a sequence of independent standard Gaussian random
variables $Y_{j}$'s so that $\mathbb{E}[X_{j}X_{k}]\leq0=\mathbb{E}[Y_{j}Y_{k}]$
for $j\neq k$. Let $c'=\frac{c}{\sqrt{2log}2}$ so that $\theta-c'>0$
and hence $0<1-\theta+c'<1$. Slepian's lemma (see Corollary 3.12,
\citep{Ledoux2013}) implies that
\begin{align*}
 & \hphantom{=\ \ }P\left(\bigcap_{1\leq j\leq2^{n}}\left\{ \lvert X_{j}\rvert\leq(1-\theta+c')\sqrt{2\log(2^{n})}\right\} \right)\\
 & \leq P\left(\bigcap_{1\leq j\leq2^{n}}\left\{ \lvert X_{j}\rvert\leq(1-\theta+c')^{1/2}\sqrt{2\log(2^{n})}\right\} \right)\\
 & \leq P\left(\bigcap_{1\leq j\leq2^{n}}\left\{ X_{j}\leq(1-\theta+c')^{1/2}\sqrt{2\log(2^{n})}\right\} \right)\\
 & \leq P\left(\bigcap_{1\leq j\leq2^{n}}\left\{ Y_{j}\leq(1-\theta+c')^{1/2}\sqrt{2\log(2^{n})}\right\} \right)\\
 & =\prod_{1\leq j\leq2^{n}}P\left(Y_{j}\leq(1-\theta+c')^{1/2}\sqrt{2\log(2^{n})}\right)\\
 & =\left[1-P\left(Y_{j}>(1-\theta+c')^{1/2}\sqrt{2\log(2^{n})}\right)\right]^{2^{n}}\vphantom{\prod_{i}}\\
 & \leq\exp\left(-\xi2^{n}\right),\vphantom{\prod_{i}}
\end{align*}
where 
\begin{align*}
\xi & =P\left(Y_{j}>(1-\theta+c')^{1/2}\sqrt{2\log(2^{n})}\right)\\
 & \geq\frac{(1-\theta+c')^{1/2}\sqrt{2\log(2^{n})}}{1+2(1-\theta+c')\log(2^{n})}\exp\left(-(1-\theta+c')\log(2^{n})\right)\\
 & \geq C2^{-n(1-\theta+c')}
\end{align*}
for $n$ sufficiently large, hence it follows that
\[
P\left(\bigcap_{1\leq j\leq2^{n}}\left\{ \lvert X_{j}\rvert\leq(1-\theta+c')\sqrt{2\log(2^{n})}\right\} \right)\leq\exp\left(-C2^{n(\theta-c')}\right).
\]
The right-hand side is a term of a convergent series, and hence by
the first Borel-Cantelli Lemma for $(p,r)$-capacity, (\ref{eq:continuity_low_bound})
follows immediately.

For the upper bound, we first notice that $g(k2^{-n})=(k2^{-n})^{H}\sqrt{2\log(\frac{2^{n}}{k})}$.
For any $\varepsilon>0$, applying lemma \ref{increment inequality}
with $\eta=(1+\varepsilon)g(k2^{-n})$, we get that 

\begin{align*}
I_{n}^{p} & =\left[c_{p,r}\left(\max_{\substack{\mathsmaller{0<k=j-i\leq2^{n\theta}}\\
\mathsmaller{0\leq i<j\leq2^{n}}
}
}\frac{\left|B_{j2^{-n}}-B_{i2^{-n}}\right|}{g(k2^{-n})}\geq1+\varepsilon\right)\right]^{p}\\
 & \leq M_{p,r}\sum_{\substack{\mathsmaller{0<k=j-i\leq2^{n\theta}}\\
\mathsmaller{0\leq i<j\leq2^{n}}
}
}\left[c_{p,r}\left(\frac{\left|B_{j2^{-n}}-B_{i2^{-n}}\right|}{g(k2^{-n})}\geq1+\varepsilon\right)\right]^{p}\\
 & \leq M_{p,r}2^{n}\sum_{\substack{1\leq k\leq2^{n\theta}}
}\left[2\sum_{l=0}^{r}\left(\frac{(1+\varepsilon)g(k2^{-n})}{p(k2^{-n})^{H}}\right)^{lp}\right](k2^{-n})^{(1+\varepsilon)^{2}}\\
 & =M_{p,r}2^{n}\sum_{\substack{1\leq k\leq2^{n\theta}}
}\left[2\sum_{l=0}^{r}\left(\frac{(1+\varepsilon)}{p}\sqrt{2\log\left(\frac{2^{n}}{k}\right)}\right)^{lp}\right](k2^{-n})^{(1+\varepsilon)^{2}}\\
 & \leq M_{p,r}2^{n(1+\theta)}\left[2\sum_{l=0}^{r}\left(\frac{(1+\varepsilon)}{p}\sqrt{2n\log2}\right)^{lp}\right]2^{-n(1-\theta)(1+\varepsilon)^{2}},
\end{align*}
where the first inequality follows from (\ref{eq:cap_p_subadditivity}).
Now we only need to pick up suitable $\theta$ such that $\sum_{n}I_{n}<\infty$.
To this end, we want $1+\theta<(1-\theta)(1+\varepsilon)^{2}$. In
fact, any 
\[
\theta\in\left(0,\frac{(1+\varepsilon)^{2}-1}{(1+\varepsilon)^{2}+1}\right)
\]
will do. The proof is complete by applying the first Borel-Cantelli
lemma for $(p,r)$-capacity and letting $\varepsilon\to0$.
\end{proof}
The upper bound (\ref{eq:continuity_upp_bound}) implies the following
result:
\begin{cor}
$(B_{t})_{t\geq0}$ is $\alpha$-H\"{o}lder-continuous for $\alpha<H$
quasi-surely with respect to the Brownian motion capacity.
\end{cor}
\begin{rem}
We regard $(B_{t})_{t\geq0}$ as a family of measurable functions
on $(\boldsymbol{W},\mathscr{F})$ with parameter $t\geq0$. What
we proved previously is that apart from a slim set, $t\to B_{t}(\omega)$
is continuous. Therefore, we can modify $\left(B_{t}\right)_{t\geq0}$
on the slim set $K$ by for example setting $B_{t}(\omega)=0$ for
all $\omega\in K$ such that the modified process is continuous, and
$K\in\mathscr{F}$ with $P(K)=0$ as $c_{p,r}$ is increasing in $p$
and $r$. From now on, we always refer $\left(B_{t}\right)_{t\geq0}$
to its continuous modification.
\end{rem}

\section{Non-differentiability}

In this part, we will generalise a very standard result based on the
argument in \citep{Dvoretzky1961} (see also \citep{Karatzas2012}
page 110), \citep{Fukushima1984} and \citep{Takeda1984}). 

\begin{thm}
Let $H\in(0,1)$. Then 
\[
\limsup_{h\downarrow0}\frac{\lvert B_{t+h}-B_{t}\rvert}{h}=\infty\quad\text{for all }t\in[0,1]\quad\text{q.s.}
\]
\end{thm}
\begin{proof}
Let
\[
A=\left\{ \limsup_{h\downarrow0}\frac{\left|B_{t+h}-B_{t}\right|}{h}<\infty\ \text{for some }t\in[0,1]\right\} .
\]
The goal is to show that $A$ is a slim set. If $\omega\in A$, then
there exists a $t\in[0,1]$, positive integers $M$ and $k$, such
that $\left|B_{t+h}(\omega)-B_{t}(\omega)\right|\leq Mh$ for all
$0\leq h\leq\frac{1}{k}$ . Therefore, we may consider
\[
A_{k,M}^{t}=\left\{ \sup_{\mathsmaller{h\in[0,\frac{1}{k}]}}\frac{\left|B_{t+h}-B_{t}\right|}{h}\leq M\right\} 
\]
where $M$ and $k$ are positive integers. Then
\[
A=\bigcup_{t\in[0,1]}\bigcup_{M=1}^{\infty}\bigcup_{k=1}^{\infty}A_{k,M}^{t}.
\]
By the sub-additivity property of $(p,r)$-capacity, it remains to
show that 
\[
c_{p,r}\Big(\bigcup_{t\in[0,1]}A_{k,M}^{t}\Big)=0
\]
for all $r\in\mathbb{N}$ and $1<p<\infty$.

Fix $r$, $p$, $k$ and $M$. For $H\in(0,1)$, take $N$ to be the
smallest integer such that $\frac{N(1-H)}{p}>1$, and divide $[0,1]$
into $n$ subintervals with $n\geq(N+1)k$. Then for all $t\in[\frac{i-1}{n},\frac{i}{n}]$,
$1\leq i\leq n$, 
\[
\frac{i+N}{n}-t\leq\frac{1}{k},
\]
which indicates that for $1\leq j\leq N$,
\begin{equation}
\frac{i+j-1}{n}-t\leq\frac{i+j}{n}-t\leq\frac{i+N}{n}-t\leq\frac{1}{k}.\label{eq:non_diff_eq1}
\end{equation}
Now if $\omega\in A_{k,M}^{t}$ with $t\in[\frac{i-1}{n},\frac{i}{n}]$,
then for each $1\leq j\leq N$, by (\ref{eq:non_diff_eq1}),
\[
\begin{aligned}\left|B_{\frac{i+j}{n}}(\omega)-B_{\frac{i+j-1}{n}}(\omega)\right| & \leq\left|B_{t+(\frac{i+j}{n}-t)}(\omega)-B_{t}(\omega)\right|+\left|B_{t}(\omega)-B_{t+(\frac{i+j-1}{n}-t)}(\omega)\right|\\
 & \leq\left[\left(\frac{i+j}{n}-t\right)+\left(\frac{i+j-1}{n}-t\right)\right]M\\
 & \leq\frac{(2j+1)M}{n}.
\end{aligned}
\]
 Therefore, if we define 
\[
C_{i,n}=\bigcap_{j=1}^{N}\left\{ n^{H}\left|B_{\frac{i+j}{n}}-B_{\frac{i+j-1}{n}}\right|\leq\frac{(2j+1)M}{n^{1-H}}\right\} ,\quad1\leq i\leq n
\]
for each $n\geq(N+1)k$, then
\[
\bigcup_{t\in[0,1]}A_{k,M}^{t}\subset\bigcup_{i=1}^{n}C_{i,n}.
\]
Therefore, it suffices to prove that $\sum_{i=1}^{n}c_{p,r}(C_{i,n})\to0$
as $n\to\infty$. 

To this end, we apply lemma \ref{capacity-probability} to bound $c_{p,r}(C_{i,n})$
from above. For each fixed $i$, set $X_{j}=n^{H}(B_{\frac{i+j}{n}}-B_{\frac{i+j-1}{n}})$,
and $\alpha_{j}=\frac{(2j+1)M}{n^{1-H}}$, $1\leq j\leq N$. By lemma
\ref{capacity-probability} with $L=\frac{1}{n}$, it follows that
\begin{equation}
\begin{aligned}\sum_{i=1}^{n}c_{p,r}(C_{i,n})\leq & \left[\sum_{l=0}^{r}\left(N^{lp}C_{H}^{lp/2}\left(\frac{M_{r}}{c}\right)^{lp}\right)\right]^{1/p}\\
 & \cdot\begin{aligned}\sum_{i=1}^{n}\left[P\left(\bigcap_{j=1}^{N}\{-\alpha_{j}-c\leq X_{j}\leq\alpha_{j}+c\}\right)\right]^{1/p},\end{aligned}
\end{aligned}
\label{eq:non_diff_eq2}
\end{equation}
where $c>0$ is a constant, $M_{r}$ and $C_{H}$ are as in lemma
\ref{capacity-probability}. Note that $(X_{1},\cdots,X_{N})$ is
a centred Gaussian random variable with covariance matrix $\Sigma$,
determined by
\[
\mathbb{E}[X_{j}X_{k}]=\frac{1}{2}\left[(k-j+1)^{2H}+(k-j-1)^{2H}\right]-(k-j)^{2H},
\]
which depends only on $j$ and $k$. $\Sigma$ is an $N\times N$
positive definite matrix independent of $n$. Therefore, the right-hand
side of (\ref{eq:non_diff_eq2}) may be computed explicitly as
\[
\begin{aligned}P\left(\bigcap_{j=1}^{N}\{|X_{j}|\leq\alpha_{j}+c\}\right) & =2^{N}\int_{0}^{\alpha_{N}+c}\cdots\int_{0}^{\alpha_{1}+c}\frac{1}{\sqrt{2\pi|\Sigma|}}\exp\left(-\frac{1}{2}\bold{x}^{T}\Sigma^{-1}\bold{x}\right)dx_{1}\cdots dx_{N}\\
 & \leq2^{N}\frac{1}{\sqrt{2\pi|\Sigma|}}\prod_{j=1}^{N}(\alpha_{j}+c)\\
 & =O(n^{-N(1-H)}),
\end{aligned}
\]
hence it follows that
\[
\sum_{i=1}^{n}\left[P\left(\bigcap_{j=1}^{N}\{-\alpha_{j}-c\leq X_{j}\leq\alpha_{j}+c\}\right)\right]^{1/p}\leq O(n\cdot n^{-N(1-H)/p})\to0
\]
as $n\to\infty$, which completes the proof.
\end{proof}

\section{Law of Iterated Logarithm}

In this section, we establish the result on law of iterated logarithm
for fBM with Hurst parameter $H\in\left(0,\frac{1}{2}\right]$ with
respect to $(p,r)$-capacity on classical Wiener space, using the
argument from \citep{Fukushima1984} together with the technical lemmas
we established in Section 3.

\begin{thm}
Let $H\in\left(0,\frac{1}{2}\right]$. Then it holds that 
\[
c_{2,1}\left(\limsup_{t\downarrow0}\frac{B_{t}}{\sqrt{2t^{2H}\log\log(1/t)}}>1\right)=0.
\]
\end{thm}
\begin{proof}
When $H=\frac{1}{2}$, the problem will be reduced to Brownian motion
case, which will be the same as in \citep{Fukushima1984} and \citep{Takeda1984}.
The rest of our proof will be similar to the argument in \citep{Fukushima1984}.
Let $h(t)=\sqrt{2t^{2H}\log\log(1/t)}$. Fix $\theta,\delta\in(0,1)$,
and set $\eta=(1+\delta)h(\theta^{n})$, $s=0$, $t=\theta^{n}$ in
Lemma \ref{lem:supremum inequality}, then it follows that
\begin{equation}
\begin{aligned} & \hphantom{=\ \ }\left[c_{2,1}\left(\sup_{0\leq u\leq\theta^{n}}B_{u}>(1+\delta)h(\theta^{n})\right)\right]^{2}\\
 & \leq\left[\left(\frac{\theta^{2nH}}{\theta^{2nH}+\theta^{n}}\right)^{2}(1+\delta^{2})\log\log(\theta^{-n})+2\right]\exp\left(-\frac{\theta^{2nH}}{\theta^{2nH}+\theta^{n}}(1+\delta)^{2}\log\log(\theta^{-n})\right)\vphantom{\left(\sup_{0\leq u\leq\theta^{n}}(\theta^{n})\right)}\\
 & \leq\left[(1+\delta^{2})\log\log(\theta^{-n})+2\right]\left(n\log(\theta^{-1})\right)^{-\frac{\theta^{2nH}}{\theta^{2nH}+\theta^{n}}(1+\delta)^{2}}\\
 & =C_{1}(\log n+C_{2})n^{-\frac{\theta^{2nH}}{\theta^{2nH}+\theta^{n}}(1+\delta)^{2}}.\vphantom{\left(\sup_{0\leq u\leq\theta^{n}}(\theta^{n})\right)}
\end{aligned}
\label{eq:LIL_eq1}
\end{equation}
For each $\theta$ and $\delta$, as $H<\frac{1}{2}$ and $\theta<1$,
there exists some $N_{0}$ such that for all $n\geq N_{0}$, 
\[
\frac{\theta^{2nH}}{\theta^{2nH}+\theta^{n}}(1+\delta)^{2}>1,
\]
so the right-hand side of (\ref{eq:LIL_eq1}) is a term of a convergent
series, and thus by the first Borel-Cantelli lemma for capacity, 
\[
{\normalcolor \sup_{0\leq u\leq\theta^{n}}B_{u}\leq(1+\delta)h(\theta^{n})\quad\text{eventually }}
\]
under $(2,1)$-capacity. The rest of proof remains the same as in
probability case.
\end{proof}
\begin{thm}
Let $(B_{t})_{t\geq0}$ be a one-dimensional fBM on $(\boldsymbol{W},\mathcal{H},P)$
with Hurst parameter $H\in\left(0,\frac{1}{2}\right]$. Then it holds
that 
\[
c_{2,1}\left(\limsup_{t\downarrow0}\frac{B_{t}}{\sqrt{2t^{2H}\log\log(1/t)}}<1\right)=0.
\]
\end{thm}
\begin{proof}
When $H=\frac{1}{2}$, the problem is reduced to Brownian motion case,
so we only need to consider the case when $H\in\left(0,\frac{1}{2}\right)$.
Let $h(t)=\sqrt{2t^{2H}\log\log(1/t)}$, and let $\theta\in(0,1)$,
define 
\[
G_{n}=\left\{ B_{\theta^{n}}-B_{\theta^{n+1}}<(1-\theta^{H})h(\theta^{n})\right\} .
\]
Our next step is to prove that
\[
c_{2,1}\left(\liminf_{n\to\infty}G_{n}\right)=0,
\]
from which we may deduce that for sufficiently large $n$,
\[
B_{\theta^{n}}-B_{\theta^{n+1}}>(1-\theta^{H})h(\theta^{n})
\]
apart from on a $(2,1)$-capacity zero set.

Write 
\[
X_{n}=\frac{B_{\theta^{n}}-B_{\theta^{n+1}}}{(\theta^{n}-\theta^{n+1})^{H}}\sim N(0,1),
\]
then by definition,
\[
G_{n}=\left\{ X_{n}<\frac{1-\theta^{H}}{(1-\theta)^{H}}\sqrt{2\log\log(\theta^{-n})}\right\} .
\]
For any integers $l\leq N$, take a decreasing sequence of real numbers
$\{a_{i}\}_{i=1}^{\infty}$ such that $a_{i}\downarrow-\infty$ as
$i\to\infty$, due to the continuity of capacity (\ref{eq:cap_lower_continuous}),
we have that 
\begin{align*}
\left[c_{2,1}\left(\bigcap_{n=l}^{N}G_{n}\right)\right]^{2} & =\left[c_{2,1}\left(\bigcap_{n=l}^{N}\left\{ X_{n}<\frac{1-\theta^{H}}{(1-\theta)^{H}}\sqrt{2\log\log(\theta^{-n})}\right\} \right)\right]^{2}\\
 & =\left[c_{2,1}\left(\bigcup_{i=1}^{\infty}\bigcap_{n=l}^{N}\left\{ a_{i}<X_{n}<\frac{1-\theta^{H}}{(1-\theta)^{H}}\sqrt{2\log\log(\theta^{-n})}\right\} \right)\right]^{2}\\
 & =\lim_{i\to\infty}\left[c_{2,1}\left(\bigcap_{n=l}^{N}\left\{ a_{i}<X_{n}<\frac{1-\theta^{H}}{(1-\theta)^{H}}\sqrt{2\log\log(\theta^{-n})}\right\} \right)\right]^{2}.
\end{align*}
Then we may apply lemma \ref{capacity-probability} to control the
intersection capacity with probability as the following:
\begin{align}
\left[c_{2,1}\left(\bigcap_{n=l}^{N}G_{n}\right)\right]^{2} & \leq\lim_{i\to\infty}\left(1+(N-l)^{2}C_{H}\left(\frac{M_{r}}{c}\right)^{2}\right)\nonumber \\
 & \hphantom{=}\cdot P\left(\bigcap_{n=l}^{N}\left\{ a_{i}-c_{n}<X_{n}<\frac{1-\theta^{H}}{(1-\theta)^{H}}\sqrt{2\log\log(\theta^{-n})}+c_{n}\right\} \right)\nonumber \\
 & \leq\left(1+(N-l)^{2}C_{H}\left(\frac{M_{r}}{c}\right)^{2}\right)\nonumber \\
 & \hphantom{=}\cdot P\left(\bigcap_{n=l}^{N}\left\{ X_{n}<\frac{1-\theta^{H}}{(1-\theta)^{H}}\sqrt{2\log\log(\theta^{-n})}+c_{n}\right\} \right).\label{eq:LIL_cap_prob}
\end{align}
When $H\in\left(0,\frac{1}{2}\right)$, the increments of fBM over
different time intervals are negatively correlated, i.e. $\mathbb{E}[X_{n}X_{m}]\leq0$.
For all $l\leq n,m\leq N$, we may take a sequence of independent
standard Gaussian random variables $\{Y_{n}\}$, and apply Slepian's
lemma to the intersection probability in the last line in (\ref{eq:LIL_cap_prob})
to obtain that 
\begin{align*}
 & \hphantom{=\ \ }P\left(\bigcap_{n=l}^{N}\left\{ X_{n}<\frac{1-\theta^{H}}{(1-\theta)^{H}}\sqrt{2\log\log(\theta^{-n})}+c_{n}\right\} \right)\\
 & \leq P\left(\bigcap_{n=l}^{N}\left\{ Y_{n}<\frac{1-\theta^{H}}{(1-\theta)^{H}}\sqrt{2\log\log(\theta^{-n})}+c_{n}\right\} \right)\\
 & =\prod_{n=l}^{N}P\left(Y_{n}<\frac{1-\theta^{H}}{(1-\theta)^{H}}\sqrt{2\log\log(\theta^{-n})}+c_{n}\right)\\
 & =\prod_{n=l}^{N}\left[1-P\left(Y_{n}\geq\frac{1-\theta^{H}}{(1-\theta)^{H}}\sqrt{2\log\log(\theta^{-n})}+c_{n}\right)\right]\\
 & \leq\exp\left[-\sum_{n=l}^{N}P\left(Y_{n}\geq\frac{1-\theta^{H}}{(1-\theta)^{H}}\sqrt{2\log\log(\theta^{-n})}+c_{n}\right)\right],
\end{align*}
where the last inequality follows from the fact that $1-x\leq e^{-x}$.
We proceed by picking up suitable $c_{n}$'s such that the left-hand
of (\ref{eq:LIL_cap_prob}) vanishes as $N$ goes to infinity. Notice
that for each $n\in[l,N]$, it holds that
\begin{align}
 & \hphantom{=\ \ }P\left(Y_{n}\geq\alpha\sqrt{2\log\log(\theta^{-n})}\right)\nonumber \\
 & =\frac{1}{\sqrt{2\pi}}\int_{\alpha\sqrt{2\log\log(\theta^{-n})}}^{\infty}e^{-\frac{x^{2}}{2}}dx\nonumber \\
 & \ge\frac{1}{\sqrt{2\pi}}\frac{\alpha\sqrt{2\log\log(\theta^{-n})}}{1+2\alpha^{2}\log\log(\theta^{-n})}\exp\left(-\alpha^{2}\log\log(\theta^{-n})\right)\nonumber \\
 & \geq\frac{1}{\sqrt{2\pi}}\frac{1}{C_{1}\sqrt{2\alpha^{2}\log\log(\theta^{-n})}}\cdot\frac{1}{n^{\alpha^{2}}(\log(\theta^{-1}))^{\alpha^{2}}}\nonumber \\
 & \geq\frac{C_{2}}{n^{\alpha^{2}}\sqrt{\log n}},\label{eq:LIL_eq2}
\end{align}
where $C_{1}$ and $C_{2}$ are positive constants. Choose suitable
$C$ and small $\beta$ such that $\log x<Cx^{\beta}$ for large $x$,
and set $c_{n}$ to be small enough such that the quantity
\[
\alpha=\frac{c_{n}}{\sqrt{2\log\log(\theta^{-n})}}+\frac{1-\theta^{H}}{(1-\theta)^{H}}
\]
satisfies $\gamma=\alpha^{2}+\frac{\beta}{2}<1$. By taking $\alpha$
equal to the above value in (\ref{eq:LIL_eq2}), we conclude that
\[
\sum_{n=l}^{N}P\left(Y_{n}\geq\frac{1-\theta^{H}}{(1-\theta)^{H}}\sqrt{2\log\log(\theta^{-n})}+c_{n}\right)\geq\sum_{n=l}^{N}\frac{C_{3}}{n^{\gamma}}\geq C_{3}(N^{1-\gamma}-l^{1-\gamma}),
\]
where $C_{3}$ is a positive constant. Therefore,
\[
\left[c_{2,1}\left(\bigcap_{n=l}^{\infty}G_{n}\right)\right]^{2}\leq\left[c_{2,1}\left(\bigcap_{n=l}^{N}G_{n}\right)\right]^{2}\leq C'(N-l)^{2}C_{H}e^{-C_{3}(N^{1-\gamma}-l^{1-\gamma})},
\]
where $C'$ is some positive constant, and $\begin{aligned}C_{H} & =\max\left\{ 2^{2H-1}-1,1\right\} \leq1\end{aligned}
$ as in lemma \ref{capacity-probability}. Since the right-hand side
of above inequality vanishes as $N$ goes to infinity, we arrive at
\[
c_{2,1}\left(\liminf_{n\to\infty}G_{n}\right)=0.
\]
\end{proof}
We are unable to extend the result to the case where $H>\frac{1}{2}$,
and we do not believe a similar result is true for this case in fact.

\section{Self-intersection of sample paths}

Recall that $\boldsymbol{W}_{0}^{d}$ consists of all $\mathbb{R}^{d}$-valued
continuous paths, started at the origin, and $(\boldsymbol{W}_{0}^{d},\mathcal{H},P)$
is the corresponding classical Wiener space. In this section, a $d$-dimensional
fBM is defined to be the functional on $(\boldsymbol{W}_{0}^{d},\mathcal{H},P)$
given by the integral
\begin{equation}
B_{t}=\int_{0}^{t}K(t,s)d\omega(s),\label{eq:d-dim fBM defn}
\end{equation}
where $\omega\in\boldsymbol{W}_{0}^{d}$ is $d$-dimensional Brownian
motion. By definition, a $d$-dimensional fBM is $d$ copies of independent
one-dimensional fBM defined as in (\ref{eq:fBM_integral_repn}) due
to the definition of multi-dimensional Brownian motion. Like in the
one-dimensional case, we take a suitable modification of $B_{t}$
such that it is quasi-surely continuous with respect to classical
Wiener capacity. 

In this section, we will study the self-avoiding property for $d$-dimensional
fBM and establish a result with respect to $(2,1)$-capacity on $(\boldsymbol{W}_{0}^{d},\mathcal{H},P)$,
following the idea by Kakutani \citep{Kakutani1944} together with
several techniques in Fukushima \citep{Fukushima1984} and Takeda
\citep{Takeda1984} to tackle with capacities.
\begin{thm}
Let $B=(B_{t})_{t\geq0}$ be the $d$-dimensional fBM defined in (\ref{eq:d-dim fBM defn})
with Hurst parameter $H$. When $H\leq\frac{1}{2}$ and $d>\frac{2}{H}+2$,
$B$ has no double point under $(2,1)$-capacity on classical Wiener
space; when $H\geq\frac{1}{2}$ and $d>6$, $B$ has no double point
under $(2,1)$-capacity.
\end{thm}
\begin{proof}
When $H=\frac{1}{2}$, the above result is proved in Fukushima \citep{Fukushima1984}
and Takeda \citep{Takeda1984}. It suffices to show that for any two
disjoint intervals $I=(s_{0},s_{1})$ and $J=(t_{0},t_{1})$ with
$s_{0}<s_{1}<t_{0}<t_{1}$,
\begin{equation}
c_{2,1}\left(B_{s}=B_{t},\text{ for some }s\in I\text{ and some }t\in J\right)=0.\label{eq:self_avd_equiv}
\end{equation}
By self-similarity property of fBM, we only need to establish the
above equality for $0\leq s_{0}<s_{1}<t_{0}<t_{1}\leq1$. Denote the
set in (\ref{eq:self_avd_equiv}) by $A$. Then for any $\eta>0$,
we may write
\[
A\subset \bigcap_{i=1}^{d}\left\{ \left|B_{s_{1}}^{i}-B_{t_{0}}^{i}\right|<2\eta\right\} \cup\bigcup_{i=1}^{d}\left\{ \sup_{s\in I}\left|B_{s_{1}}^{i}-B_{s}^{i}\right|>\eta\right\} \cup\bigcup_{i=1}^{d}\left\{ \sup_{t\in J}\left|B_{t}^{i}-B_{t_{0}}^{i}\right|>\eta\right\} ,
\]
where $B^{i}$ is the $i$-th component of $B$. It thus follows from
sub-additivity property of capacity that 
\[
\begin{aligned}c_{2,1}(A)\leq &\ c_{2,1}\left(\bigcap_{i=1}^{d}\left\{ \left|B_{s_{1}}^{i}-B_{t_{0}}^{i}\right|<2\eta\right\} \right)\\
 & +\sum_{i=1}^{d}c_{2,1}\left(\sup_{s\in I}\left|B_{s_{1}}^{i}-B_{s}^{i}\right|>\eta\right)\\
 & +\sum_{i=1}^{d}c_{2,1}\left(\sup_{t\in J}\left|B_{t}^{i}-B_{t_{0}}^{i}\right|>\eta\right).
\end{aligned}
\]
Applying lemma \ref{capacity-probability} with $c=c_{i}=\eta$, $i=1,2,\cdots,d$,
we obtain that
\[
\begin{aligned}c_{2,1}\left(\bigcap_{i=1}^{d}\left\{ -2\eta<B_{s_{1}}^{i}-B_{t_{0}}^{i}<2\eta\right\} \right)\leq & \left(1+d^{2}C_{H}\left(\frac{M}{\eta}\right)^{2}\right)\\
 & \cdot P\left(\bigcap_{i=1}^{d}\left\{ \left|B_{s_{1}}^{i}-B_{t_{0}}^{i}\right|<3\eta\right\} \right),
\end{aligned}
\]
where $C_{H}=\max\left\{ 2^{2H-1}-1,1\right\} \leq1$, and $M$ is
some positive constant. Therefore, 
\begin{align*}
 & c_{2,1}\left(\bigcap_{i=1}^{d}\left\{ \left|B_{s_{1}}^{i}-B_{t_{0}}^{i}\right|<2\eta\right\} \right)\\
\leq & \left(1+d^{2}\left(\frac{M}{\eta}\right)^{2}\right)\prod_{i=1}^{d}P\left(\left|B_{s_{1}}^{i}-B_{t_{0}}^{i}\right|<3\eta\vphantom{\left(\frac{M}{\eta}\right)^{2}}\right)\\
= & \left(1+d^{2}\left(\frac{M}{\eta}\right)^{2}\right)\left[\frac{1}{\sqrt{2\pi(t_{0}-s_{1})^{2H}}}\int_{-3\eta}^{3\eta}\exp\left(-\frac{x^{2}}{2(t_{0}-s_{1})^{2H}}\right)dx\right]^{d}\\
\leq & \left(1+d^{2}\left(\frac{M}{\eta}\right)^{2}\right)\left(\frac{6\eta}{\sqrt{2\pi\left(d(I,J)\right)^{2H}}}\right)^{d},
\end{align*}
where $d(I,J)=t_{0}-s_{1}$ denotes the distance between these two
intervals. Also, applying lemma \ref{lem:supremum inequality}, it
follows that 
\[
\begin{aligned}c_{2,1}\left(\sup_{s\in I}\left|B_{s_{1}}^{i}-B_{s}^{i}\right|>\eta\right) & \leq\sqrt{\frac{\eta^{2}(s_{1}-s_{0})^{2H}}{\left[\gamma_{H}(s_{1}-s_{0})^{2H}+(s_{1}-s_{0})\right]^{2}}+4}\\
 & \hphantom{=}\cdot\exp\left(-\frac{\eta^{2}}{4[\gamma_{H}(s_{1}-s_{0})^{2H}+(s_{1}-s_{0})]}\right)\\
 & =\sqrt{\frac{\eta^{2}|I|^{2H}}{\left[\gamma_{H}|I|^{2H}+|I|\right]^{2}}+4}\cdot\exp\left(-\frac{\eta^{2}}{4[\gamma_{H}|I|^{2H}+|I|]}\right),
\end{aligned}
\]
where $|I|=s_{1}-s_{0}$ denote the length of $I$. Accordingly,
\[
c_{2,1}\left(\sup_{t\in J}\left|B_{t}^{i}-B_{t_{0}}^{i}\right|>\eta\right)\leq\sqrt{\frac{\eta^{2}|J|^{2H}}{\left[\gamma_{H}|J|^{2H}+|J|\right]^{2}}+4}\exp\left(-\frac{\eta^{2}}{4[\gamma_{H}|J|^{2H}+|J|]}\right),
\]
with $|J|=t_{1}-t_{0}$, the length of interval $J$. 

Divide $I$ and $J$ into $k$ subintervals evenly, i.e. $I=\bigcup_{m=1}^{k}I_{m}$,
$J=\bigcup_{l=1}^{k}J_{l}$, $I_{m}$ and $J_{l}$ are disjoint for
all $1\leq m,l\leq k$ and $|I_{m}|=|I|/k$, $|J_{l}|=|J|/k$. By
sub-additivity and above,
\begin{align*}
c_{2,1}(A)\leq & \sum_{m=1}^{k}\sum_{l=1}^{k}c_{2,1}\Big(B_{s}=B_{t},\text{ for some }s\in I_{m}\text{ and some }t\in J_{l}\Big)\vphantom{\left(\frac{6\eta}{\sqrt{\left(a\right)^{2H}}}\right)}\\
\leq & \sum_{m=1}^{k}\sum_{l=1}^{k}\left[\left(1+d^{2}\left(\frac{M}{\eta}\right)^{2}\right)\left(\frac{6\eta}{\sqrt{2\pi\left(d(I_{m},J_{l})\right)^{2H}}}\right)^{d}\right.\\
 & +d\sqrt{\frac{\eta^{2}|I_{m}|^{2H}}{\left(\gamma_{H}|I_{m}|^{2H}+|I_{m}|\right)^{2}}+4}\exp\left(-\frac{\eta^{2}}{4\left(\gamma_{H}|I_{m}|^{2H}+|I_{m}|\right)}\right)\\
 & \left.+d\sqrt{\frac{\eta^{2}|J_{l}|^{2H}}{\left(\gamma_{H}|J_{l}|^{2H}+|J_{l}|\right)^{2}}+4}\exp\left(-\frac{\eta^{2}}{4\left(\gamma_{H}|J_{l}|^{2H}+|J_{l}|\right)}\right)\right]\\
\leq &\ k^{2}\left[\left(1+d^{2}\left(\frac{M}{\eta}\right)^{2}\right)\left(\frac{6\eta}{\sqrt{2\pi\left(d(I,J)\right)^{2H}}}\right)^{d}\right.\\
 & +d\sqrt{\frac{\eta^{2}k^{2H}}{\gamma_{H}^{2}|I|^{2H}}+4}\exp\left(-\frac{\eta^{2}}{4\left(\gamma_{H}|I|^{2H}k^{-2H}+|I|k^{-1}\right)}\right)\\
 & \left.+d\sqrt{\frac{\eta^{2}k^{2H}}{\gamma_{H}^{2}|J|^{2H}}+4}\exp\left(-\frac{\eta^{2}}{4\left(\gamma_{H}|J|^{2H}k^{-2H}+|J|k^{-1}\right)}\right)\right].
\end{align*}
Set $\eta=k^{-\sigma}$, then according to the previous estimate,
when $k$ is sufficiently large and $H<\frac{1}{2}$, it holds that
\begin{equation}
c_{2,1}(A)\leq C_{1}\left(k^{2-\sigma(d-2)}+k^{(H-\sigma)+2}e^{-Ck^{2(H-\sigma)}}\right),\label{eq:self_avd_est}
\end{equation}
where $C_{1}$ is some constant. Notice that when $\frac{2}{d-2}<\sigma<H$,
the expression on the right-hand side of (\ref{eq:self_avd_est})
vanishes as $k$ tends to infinity. This implies that if such a $\sigma$
exists, then $B$ has no double point under $(2,1)$-capacity, which
only requires $\frac{2}{d-2}<H$, i.e. $d>\frac{2}{H}+2$.

On the other hand, when $H>\frac{1}{2}$, by setting $\eta=k^{-\sigma}$,
we get that 
\[
c_{2,1}(A)\leq C_{2}\left(k^{2-\sigma(d-2)}+k^{(H-\sigma)+2}e^{-Ck^{2(1-\sigma)}}\right)
\]
when $k$ is sufficiently large, where $C_{2}$ is a constant. Therefore,
in order to guarantee that the right-hand side vanishes as $k$ tends
to infinity, we require $\frac{2}{d-2}<\sigma<\frac{1}{2}$, which
forces $d>6$.
\end{proof}
\begin{rem}
For $d$-dimensional Brownian motion, absence of double points under
$(2,1)$-capacity was proved by Fukushima in \citep{Fukushima1984}.
According to Lyons \citep{Lyons1986}, the critical dimension for
such property is $d=6$. Due to lack of tools such as potential theory,
the critical dimension of self-avoiding property for fBM remains,
we believe, an open question even in probability context.
\end{rem}

\appendix

\section{Appendix}

In this appendix, we provide a proof for lemma \ref{lem:Malliavin deriv}
in this section. The proof is a modification of the proof for Proposition
3.1 in Decreusefond and \"{U}st\"{u}nel \citep{Decreusefond1999}.
The following elementary estimate, which will be used in the proof,
taken from Theorem 3.2 in \citep{Decreusefond1999}: for any $H\in(0,1)$,
there exists a constant $c_{H}$ such that
\begin{equation}
K(t,r)\leq c_{H}r^{-\left|H-\frac{1}{2}\right|}(t-r)^{-\left(\frac{1}{2}-H\right)_{+}}\mathds{1}_{[0,t]}(r)\label{eq:K_estimate}
\end{equation}
 for any $t>r\geq0$, where $x_{+}=\max(x,0)$. 

\smallskip

\emph{Proof of Lemma} \ref{lem:Malliavin deriv}. For each fixed $t>0$,
denote $u_{t}(s)=K(t,s)\mathds{1}_{[0,t]}(s)$ for simplicity, and
set for $n\in\mathbb{N}$,
\[
u_{t}^{(n)}(s)=\sum_{i=0}^{2^{n}-1}\frac{2^{n}}{t}\left(\int_{i2^{-n}t}^{(i+1)2^{-n}t}u_{t}(r)dr\right)\mathds{1}_{(i2^{-n}t,(i+1)2^{-n}t]}(s).
\]
Then $u_{t}$ and $u_{t}^{(n)}$, $n\in\mathbb{N}$, belong to $L^{2}([0,\infty))$.
For convenience, let 
\[
F_{i}^{t,(n)}=\frac{2^{n}}{t}\left(\int_{i2^{-n}t}^{(i+1)2^{-n}t}u_{t}(r)dr\right),\quad0\leq i\leq2^{n}-1.
\]

We want to apply the dominated convergence theorem to show that for
each $t>0$, $u_{t}^{(n)}\to u_{t}$ in $L^{2}([0,\infty))$. Our
first step is to find a control function of $\{u_{t}^{(n)}\}$ in
$L^{2}([0,\infty))$. Notice that $u_{t}^{(n)}(s)$ vanishes outside
of $(0,t]$, and it is defined to be a step function inside $(0,t]$,
so we only need to check that on each ``step'', i.e. $s\in(i2^{-n}t,(i+1)2^{-n}t]$,
$0\leq i\leq2^{n-1}$, $u_{t}^{(n)}(s)$ is controlled. 

When $H>\frac{1}{2}$, for each $s\in(i2^{-n}t,(i+1)2^{-n}t]$, $0\leq i\leq2^{n-1}$,
by the estimate in (\ref{eq:K_estimate}),
\begin{align*}
\left\vert u_{t}^{(n)}(s)\right\vert  & =\frac{2^{n}}{t}\int_{i2^{-n}t}^{(i+1)2^{-n}t}K(t,r)dr\\
 & \leq\frac{2^{n}}{t}\int_{i2^{-n}t}^{(i+1)2^{-n}t}c_{H}r^{\frac{1}{2}-H}dr\\
 & =c_{H}'\left(\frac{t}{2^{n}}\right)^{\frac{1}{2}-H}\left[\left(\mathsmaller{\frac{i+1}{2^{n}}}\right)^{\frac{3}{2}-H}-\left(\mathsmaller{\frac{i}{2^{n}}}\right)^{\frac{3}{2}-H}\right]\\
 & \leq c_{H}'t^{\frac{1}{2}-H}\left[(i+1)\left(\mathsmaller{\frac{i+1}{2^{n}}}\right)^{\frac{1}{2}-H}-i\left(\mathsmaller{\frac{i+1}{2^{n}}}\right)^{\frac{1}{2}-H}\right]\\
 & =c_{H}'t^{\frac{1}{2}-H}\left(\mathsmaller{\frac{i+1}{2^{n}}}\right)^{\frac{1}{2}-H}\\
 & \leq c_{H}'t^{\frac{1}{2}-H}s^{\frac{1}{2}-H},\vphantom{\frac{a}{b}^{a}}
\end{align*}
where $c_{H}'=c_{H}\left(\frac{3}{2}-H\right)^{-1}$. This implies
that when $H>\frac{1}{2}$, we may take the control function to be
$c_{H}'t^{\frac{1}{2}-H}s^{\frac{1}{2}-H}\mathds{1}_{(0,t]}(s)$. 

When $H<\frac{1}{2}$, similar to above, we have that by (\ref{eq:K_estimate}),
\begin{align*}
\left\vert u_{t}^{(n)}(s)\right\vert  & =\frac{2^{n}}{t}\int_{i2^{-n}t}^{(i+1)2^{-n}t}K(t,r)dr\vphantom{\left(\frac{i}{2^{n}}\right)^{H}}\\
 & \leq c_{H}\frac{2^{n}}{t}\int_{s}^{(i+1)2^{-n}t}r^{H-\frac{1}{2}}(t-r)^{H-\frac{1}{2}}dr+c_{H}\frac{2^{n}}{t}\int_{i2^{-n}t}^{s}r^{H-\frac{1}{2}}(t-r)^{H-\frac{1}{2}}dr\vphantom{\left(\frac{i}{2^{n}}\right)^{H}}\\
 & \leq c_{H}\frac{2^{n}}{t}s^{H-\frac{1}{2}}\int_{i2^{-n}t}^{(i+1)2^{-n}t}(t-r)^{H-\frac{1}{2}}dr+c_{H}\frac{2^{n}}{t}(t-s)^{H-\frac{1}{2}}\int_{i2^{-n}t}^{(i+1)2^{-n}t}r^{H-\frac{1}{2}}dr\vphantom{\left(\frac{i}{2^{n}}t\right)^{H}}\\
 & =c_{H}'\frac{2^{n}}{t}s^{H-\frac{1}{2}}\left[\left(\mathsmaller{t-\frac{i}{2^{n}}t}\right)^{H+\frac{1}{2}}-\left(\mathsmaller{t-\frac{i+1}{2^{n}}t}\right)^{H+\frac{1}{2}}\right]\vphantom{\left(\frac{i}{2^{n}}\right)^{H}}\\
 & \hphantom{=}+c_{H}'\frac{2^{n}}{t}(t-s)^{H-\frac{1}{2}}\left[\left(\mathsmaller{\frac{i+1}{2^{n}}t}\right)^{H+\frac{1}{2}}-\left(\mathsmaller{\frac{i}{2^{n}}t}\right)^{H+\frac{1}{2}}\right]\vphantom{\left(\frac{i}{2^{n}}\right)^{H}}\\
 & \leq c_{H}'\frac{2^{n}}{t}s^{H-\frac{1}{2}}\left[\left(\mathsmaller{t-\frac{i}{2^{n}}t}\right)\left(\mathsmaller{t-\frac{i}{2^{n}}t}\right)^{H-\frac{1}{2}}-\left(\mathsmaller{t-\frac{i+1}{2^{n}}t}\right)\left(\mathsmaller{t-\frac{i}{2^{n}}t}\right)^{H-\frac{1}{2}}\right]\vphantom{\left(\frac{i}{2^{n}}\right)^{H}}\\
 & \hphantom{=}+c_{H}'\frac{2^{n}}{t}(t-s)^{H-\frac{1}{2}}\left[\left(\mathsmaller{\frac{i+1}{2^{n}}t}\right)\left(\mathsmaller{\frac{i+1}{2^{n}}t}\right)^{H-\frac{1}{2}}-\left(\mathsmaller{\frac{i}{2^{n}}t}\right)\left(\mathsmaller{\frac{i+1}{2^{n}}t}\right)^{H-\frac{1}{2}}\right]\vphantom{\left(\frac{i}{2^{n}}\right)^{H}}\\
 & =c_{H}'s^{H-\frac{1}{2}}\left(\mathsmaller{t-\frac{i}{2^{n}}t}\right)^{H-\frac{1}{2}}+c_{H}'(t-s)^{H-\frac{1}{2}}\left(\mathsmaller{\frac{i+1}{2^{n}}t}\right)^{H-\frac{1}{2}}\vphantom{\left(\frac{i}{2^{n}}\right)^{H}}\\
 & \leq2c_{H}'s^{H-\frac{1}{2}}(t-s)^{H-\frac{1}{2}}.\vphantom{\left(\frac{i}{2^{n}}\right)^{H}}
\end{align*}
Therefore, when $H<\frac{1}{2}$, the control function is $2c_{H}'s^{H-\frac{1}{2}}(t-s)^{H-\frac{1}{2}}\mathds{1}_{(0,t]}(s)$,
which is an element of $L^{2}([0,\infty))$. 

On the other hand,
\[
u_{t}^{(n)}(s)=\frac{\int_{0}^{(i+1)2^{-n}t}u_{t}(r)dr-\int_{0}^{i2^{-n}t}u_{t}(r)dr}{2^{-n}t}\to u_{t}(s)
\]
as $n$ tends to infinity due to the continuity of $u_{t}(s)$ on
$(0,t)$. Now we may apply the dominated convergence theorem and conclude
that $u_{t}^{(n)}\to u_{t}$ in $L^{2}([0,\infty))$. 

For fixed $t\in[0,1]$, set
\begin{equation}
B_{t}^{(n)}(\omega)=\begin{cases}
\sum_{i=0}^{2^{n}-1}F_{i}^{t,(n)}\left(\omega_{(i+1)2^{-n}t}-\omega_{i2^{-n}t}\right), & 0<t\leq1,\\
0,\vphantom{\sum_{i=0}^{2^{n}-1}} & t=0.
\end{cases}\label{eq:Bn_martingale}
\end{equation}
Let $\mathscr{G}=(\mathscr{G}_{n})_{n\geq0}$, where $\mathcal{\mathscr{G}}_{n}=\sigma\left\{ \omega_{i2^{-n}t},0\leq i\leq2^{n}\right\} $
is the $\sigma$-algebra generated by $\omega_{i2^{-n}t}$'s, $0\leq i\leq2^{n}$.
Then $(B_{t}^{(n)})_{n\in\mathbb{N}}$ is a discrete martingale with
respect to this filtration $\mathcal{\mathscr{G}}$. This was observed
by Decreusefond and \"{U}st\"{u}nel\citep{Decreusefond1999}. 

We claim that $(B_{t}^{(n)})_{n\in\mathbb{N}}$ defined in (\ref{eq:Bn_martingale})
i \ref{lem:Malliavin deriv} is a discrete martingale with respect
to $\mathcal{\mathscr{G}}$, where $\mathscr{G}=(\mathscr{G}_{n})_{n\geq0}$,
the $\sigma$-algebra generated by $\omega_{i2^{-n}t}$'s, $0\leq i\leq2^{n}$.
The proof of this claim relies on the fact that for a standard Brownian
motion $\omega_{t}$ and any $0\leq t_{0}<t_{1}<\cdots<t_{n}$, 
\begin{equation}
\mathbb{E}\left[\omega_{t_{i}}\left|\omega_{t_{0}},\omega_{t_{1}},\cdots\omega_{t_{i-1}},\omega_{t_{i+1}},\omega_{t_{n}}\vphantom{\frac{t_{i+1}}{t_{i+1}}}\right.\right]=\frac{t_{i+1}-t_{i}}{t_{i+1}-t_{i-1}}\omega_{t_{i-1}}+\frac{t_{i}-t_{i-1}}{t_{i+1}-t_{i-1}}\omega_{t_{i+1}}.\label{eq:Middlepoint_con_exp}
\end{equation}
To verify (\ref{eq:Middlepoint_con_exp}), one only needs to spot
that for each $i$ and $n$, 
\[
X_{i}:=\omega_{t_{i}}-\frac{t_{i+1}-t_{i}}{t_{i+1}-t_{i-1}}\omega_{t_{i-1}}-\frac{t_{i}-t_{i-1}}{t_{i+1}-t_{i-1}}\omega_{t_{i+1}}
\]
is independent of $\sigma(\omega_{t_{0}},\omega_{t_{1}},\cdots\omega_{t_{i-1}},\omega_{t_{i+1}},\omega_{t_{n}})$.
Indeed, for any $0\leq j<i\leq n$,
\[
\begin{aligned}\mathbb{E}\left[X_{i}\omega_{t_{j}}\right] & =\mathbb{E}\left[\omega_{t_{i}}\omega_{t_{j}}\right]-\frac{t_{i+1}-t_{i}}{t_{i+1}-t_{i-1}}\mathbb{E}\left[\omega_{t_{i-1}}\omega_{t_{j}}\right]-\frac{t_{i}-t_{i-1}}{t_{i+1}-t_{i-1}}\mathbb{E}\left[\omega_{t_{i+1}}\omega_{t_{j}}\right]\\
 & =t_{j}-\frac{t_{i+1}-t_{i}}{t_{i+1}-t_{i-1}}t_{j}-\frac{t_{i}-t_{i-1}}{t_{i+1}-t_{i-1}}t_{j}\\
 & =0,\vphantom{\frac{t_{i+1}}{t_{i+1}}}
\end{aligned}
\]
and one may verify $X_{i}$ and $\omega_{t_{j}}$ are independent
via similar computation when $0<i<j\leq n$. Thus $\omega_{t_{i}}$
is independent of all linear combinations of $\omega_{t_{0}},\omega_{t_{1}},\cdots\omega_{t_{i-1}},\omega_{t_{i+1}},\omega_{t_{n}}$,
and hence $\sigma(\omega_{t_{0}},\omega_{t_{1}},\cdots\omega_{t_{i-1}},\omega_{t_{i+1}},\omega_{t_{n}})$.
Therefore, we get that
\[
\begin{aligned} & \hphantom{=\ \ }\mathbb{E}\left[\omega_{t_{i}}\left|\omega_{t_{0}},\omega_{t_{1}},\cdots\omega_{t_{i-1}},\omega_{t_{i+1}},\omega_{t_{n}}\vphantom{\frac{t_{i+1}}{t_{i+1}}}\right.\right]\\
 & =\mathbb{E}\left[X_{i}+\frac{t_{i+1}-t_{i}}{t_{i+1}-t_{i-1}}\omega_{t_{i-1}}+\frac{t_{i}-t_{i-1}}{t_{i+1}-t_{i-1}}\omega_{t_{i+1}}\left|\omega_{t_{0}},\omega_{t_{1}},\cdots\omega_{t_{i-1}},\omega_{t_{i+1}},\omega_{t_{n}}\vphantom{\frac{t_{i+1}}{t_{i+1}}}\right.\right]\\
 & =\frac{t_{i+1}-t_{i}}{t_{i+1}-t_{i-1}}\omega_{t_{i-1}}+\frac{t_{i}-t_{i-1}}{t_{i+1}-t_{i-1}}\omega_{t_{i+1}}.
\end{aligned}
\]
For each $1\leq i\leq2^{n}-1$, if $i$ is odd, then we may write
$i=2k+1$, $0\leq k\leq2^{-n+1}-1$, and thus by (\ref{eq:Middlepoint_con_exp}),
\[
\begin{aligned}\mathbb{E}\left[\omega_{(i+1)2^{-n}t}-\omega_{i2^{-n}t}|\mathcal{\mathscr{G}}_{n-1}\right] & =\mathbb{E}\left[\omega_{(k+1)2^{-n+1}t}-\omega_{(2k+1)2^{-n}t}|\mathcal{\mathscr{G}}_{n-1}\right]\\
 & =\frac{1}{2}\omega_{(k+1)2^{-n+1}t}-\frac{1}{2}\omega_{k2^{-n+1}t}.
\end{aligned}
\]
If $i$ is even, write $i=2k$ for $0\leq k\leq2^{-n+1}-1$, then
it holds that 
\[
\begin{aligned}\mathbb{E}\left[\omega_{(i+1)2^{-n}t}-\omega_{i2^{-n}t}|\mathscr{G}_{n-1}\right] & =\mathbb{E}\left[\omega_{(2k+1)2^{-n}t}-\omega_{k2^{-n+1}t}|\mathcal{\mathscr{G}}_{n-1}\right]\\
 & =\frac{1}{2}\omega_{(k+1)2^{-n+1}t}-\frac{1}{2}\omega_{k2^{-n+1}t}.
\end{aligned}
\]
Therefore, by the definition of $F_{i}^{t,(n)}$, we conclude that
\begin{align*}
\mathbb{E}\left[B_{t}^{(n)}\left|\mathcal{\mathscr{G}}_{n-1}\vphantom{B_{t}^{(n)}}\right.\right]= & \sum_{i=0}^{2^{n}-1}F_{i}^{t,(n)}\mathbb{E}\left[\omega_{(i+1)2^{-n}t}-\omega_{i2^{-n}t}\left|\mathcal{\mathscr{G}}_{n-1}\vphantom{B_{t}^{(n)}}\right.\right]\\
= & \sum_{k=0}^{2^{n-1}-1}F_{2k+1}^{t,(n)}\left(\frac{1}{2}\omega_{(k+1)2^{-n+1}t}-\frac{1}{2}\omega_{k2^{-n+1}t}\right)\\
 & +\sum_{k=0}^{2^{n-1}-1}F_{2k}^{t,(n)}\left(\frac{1}{2}\omega_{(k+1)2^{-n+1}t}-\frac{1}{2}\omega_{k2^{-n+1}t}\right)\\
= & \sum_{k=0}^{2^{n-1}-1}\frac{2^{-n+1}}{t}\left(\omega_{(k+1)2^{-n+1}t}-\omega_{k2^{-n+1}t}\right)\\
 & \cdot\left(\int_{k2^{-n+1}t}^{(2k+1)2^{-n}t}u_{t}(s)ds+\int_{(2k+1)2^{-n}t}^{(k+1)2^{-n+1}t}u_{t}(s)ds\right)\vphantom{\sum_{i=0}^{2^{n}-1}}\\
= & B_{t}^{(n-1)}.\vphantom{\sum_{i=0}^{2^{n}-1}}
\end{align*}
For $p\in(1,\infty)$, because the increments of $\omega_{t}$ over
different time intervals are independent, and $B_{t}^{(n)}$ is contained
in the first Wiener chaos, by (2.3) from Lemma 2.2 in \citep{Boedihardjo2016}
with $N=1$,
\begin{align*}
\big\Vert B_{t}^{(n)}\big\Vert_{p} & \leq2\sqrt{p-1}\big\Vert B_{t}^{(n)}\big\Vert_{2}\vphantom{\left(\sum_{i=1}^{2^{n}-1}\right)^{\frac{1}{2}}}\\
 & =2\sqrt{p-1}\left[\sum_{i=1}^{2^{n}-1}\left(\frac{2^{n}}{t}\right)^{2}\left(\int_{(i-1)2^{-n}t}^{i2^{-n}t}u_{t}(s)ds\right)^{2}\mathbb{E}\left[\left(\omega_{(i+1)2^{-n}}-\omega_{i2^{-n}t}\right)^{2}\right]\right]^{\frac{1}{2}}\\
 & =2\sqrt{p-1}\left[\sum_{i=1}^{2^{n}-1}\frac{2^{n}}{t}\left(\int_{(i-1)2^{-n}t}^{i2^{-n}t}u_{t}(s)ds\right)^{2}\right]^{\frac{1}{2}}\\
 & \leq2\sqrt{p-1}\left(\sum_{i=1}^{2^{n}-1}\int_{(i-1)2^{-n}t}^{i2^{-n}t}u_{t}^{2}(s)ds\right)^{\frac{1}{2}}\\
 & =2\sqrt{p-1}t^{H},\vphantom{\left(\sum_{i=1}^{2^{n}-1}\right)^{\frac{1}{2}}}
\end{align*}
and hence $\sup_{n\in\mathbb{N}}\mathbb{E}[|B_{t}^{(n)}|^{p}]<\infty$.
It thus follows from the martingale convergence theorem that $(B_{t}^{(n)})_{n\in\mathbb{N}}$
conver \ref{lem:Malliavin deriv} ige to $B_{t}$ in $L^{p}(\boldsymbol{W})$,
and $B_{t}$ is a Gaussian random variable with mean zero and covariance
given by 
\[
\begin{aligned}\lim_{n\to\infty}\mathbb{E}\left[B_{s}^{(n)}B_{t}^{(n)}\right] & =\lim_{n\to\infty}\mathbb{E}\left[\left(\int_{0}^{\infty}u_{t}^{(n)}(r)d\omega_{r}\right)\left(\int_{0}^{\infty}u_{s}^{(n)}(r)d\omega_{r}\right)\right]\\
 & =\lim_{n\to\infty}\int_{0}^{\infty}u_{t}^{(n)}(r)u_{s}^{(n)}(r)dr\\
 & =\int_{0}^{s\wedge t}K(t,r)K(s,r)dr\\
 & =R(s,t)\vphantom{\int_{0}^{s\wedge t}}
\end{aligned}
\]
for any $s,t>0$. In particular, the variance of $B_{t}$ is given
by $\lim_{n\to\infty}\mathbb{E}[|B_{t}^{(n)}|^{2}]=t^{2H}$.

Now by the definition of Malliavin derivative, for $t>0$,
\[
DB_{t}^{(n)}(s)=\int_{0}^{s}u_{t}^{(n)}(v)dv,
\]
and higher-order derivatives of $B_{t}^{(n)}$ all vanish. We have
already proved that $B_{t}^{(n)}\to B_{t}$ in $L^{p}(\boldsymbol{W})$
and $u_{t}^{(n)}\to u_{t}$ in $L^{2}([0,\infty))$, so for any $r\in\mathbb{N}$
and $p\in(1,\infty)$, as
\begin{align*}
\lVert B_{t}^{(n)}-B_{t}^{(m)}\rVert_{\mathbb{D}_{r}^{p}}= & \left(\mathbb{E}\left[|B_{t}^{(n)}-B_{t}^{(m)}|^{p}\right]+\mathbb{E}\left[\left|\lVert DB_{t}^{(n)}-DB_{t}^{(m)}\rVert_{\mathcal{H}}\right|^{p}\right]\right)^{1/p}\\
= & \left(\mathbb{E}\left[|B_{t}^{(n)}-B_{t}^{(m)}|^{p}\right]+\mathbb{E}\left[\left|\lVert u_{t}^{(n)}-u_{t}^{(m)}\rVert_{L^{2}([0,\infty))}\right|^{p}\right]\right)^{1/p},
\end{align*}
we obtain that $(B_{t}^{(n)})_{n\in\mathbb{N}}$ is Cauchy in $\mathbb{D}_{r}^{p}$.
By the completeness of $\mathbb{D}_{r}^{p}$, this sequence tends
to a limit random variable in $\mathbb{D}_{r}^{p}$ as $n$ goes to
infinity. Now by the definition of $\lVert\cdot\rVert_{\mathbb{D}_{r}^{p}}$,
this convergence implies convergence in $L^{p}(\boldsymbol{W})$,
and by the uniqueness of limit, this random variable must coincide
with $B_{t}$. Moreover, 
\[
DB_{t}(s)=\int_{0}^{s\wedge t}K(t,u)du,
\]
where $DB_{t}\in\mathcal{H}$ is the Malliavin derivative of $B_{t}$
with respect to Brownian motion, and its higher order Malliavin derivatives
all vanish. 

\bibliographystyle{apalike}

\end{document}